\documentclass[oneside,reqno]{amsart}
\usepackage[bottom=0.5in]{geometry}
\usepackage{graphicx}%
\usepackage{multirow}%
\usepackage{amsmath,amssymb,amsfonts}%
\usepackage{amsthm}%
\usepackage{newtxmath}
\usepackage[T1]{fontenc}
\usepackage{fix-cm}
\usepackage[title]{appendix}%
\usepackage{xcolor}%
\usepackage{textcomp}%
\usepackage{manyfoot}%
\usepackage{booktabs}%
\usepackage{algorithm}%
\usepackage{algorithmicx}%
\usepackage{algpseudocode}%
\usepackage{listings}%
\usepackage{lmodern}
\usepackage{thmtools}
\usepackage{esint}

\usepackage{breqn}
\usepackage{bbm}
\usepackage{setspace}
\usepackage{BOONDOX-uprscr}
\usepackage[english]{babel}
\usepackage[utf8]{inputenc}
\usepackage{csquotes}
\usepackage{mathtools}
\usepackage{hyperref}
\hypersetup{
    colorlinks=true,
    linkcolor=black,
    filecolor=magenta,
    urlcolor=cyan,
}
\usepackage[backend=biber]{biblatex}
\AtBeginDocument{%
  \makeatletter
  \let\oldautoref\autoref
  \renewcommand{\autoref}[1]{\textit{\oldautoref{#1}}}%
  \makeatother
}

\allowdisplaybreaks
\DeclareMathAlphabet{\mathpzc}{OT1}{pzc}{m}{it}
\DeclarePairedDelimiter\floor{\lfloor}{\rfloor}
\newtheorem{theorem}{Theorem}
\newtheorem{lemma}{Lemma}

\newtheorem*{thm1.1}{\bf Theorem 1.1}
\newtheorem*{thm1.2}{\bf Theorem 1.2 (Unconditional)}
\newtheorem*{lem5.3}{\bf Lemma 5.3}

\theoremstyle{definition}
\newtheorem{notation}[theorem]{Notation}
\theoremstyle{remark}
\newtheorem{remark}{ \bf Remark}
\newcommand{\dblprime}{^{\prime\prime}}
\numberwithin{equation}{section}
 
 \newcommand{\lr}[1]{\left(#1\right)}

\begin{document}
\title[]{Sum of the $\textrm{\textrm{GL}}(3)$ Fourier coefficients over mixed powers}

\author[1]{Himanshi Chanana}

\author[2]{Saurabh Kumar Singh}

\email{hchanana20@gmail.com,  skumar.bhu12@gmail.com}


\address{ Department of Mathematics and Statistics,
Indian Institute of Technology Kanpur, 
Kalyanpur, Kanpur Nagar-208016, India.} 







\begin{abstract}
Let $\mathcal{A}(n)$ be the $(1,n)$-th Fourier coefficients of $\textrm{\textrm{SL}}(3,\mathbb{Z})$ Hecke-Maass cusp form, denoted as $A(1,n)$ or the triple divisor function, denoted as $d_3(n)$. Let $k \geqslant3$ be an integer. In this paper, we establish an asymptotic formula for the sum
\begin{equation*}
	 \mathop{\sum}_{\substack{1 \leqslant n_1, n_2 \leqslant X^{1/2} \\ 1 \leqslant n_3 \leqslant X^{1/k}}} \mathcal{A}(Q(n_1,n_2) + n_3^k)\mathsf{a}(n_3),
\end{equation*}
where $\mathsf{a}(n)$ is either von-Mangoldt function or identity function, and $Q(x,y) \in \mathbb{Z}[x,y]$ is a binary quadratic polynomial. When $\mathcal{A}(n)=A(1,n)$, then $\mathsf{a}(n)$ can be any bounded arithmetical function.
\end{abstract}

\maketitle

\section{Introduction}
Let $a(n)$ be an arithmetical function with the generating Dirichlet series $G(s)$. Assume $G(s)$ has a meromorphic continuation to the whole $s$-plane. If $G(s)$ has a pole, the summatory function, which is defined as $S(X):= \sum_{n \leqslant X} a(n)$, will have a main term of size $X$. 
At the core of numerous central problems in number theory lies the quest for an asymptotic formula that can provide the most precise approximation for the summatory function $S(X)$. Many studies on the averages of various traditional arithmetical functions have been conducted in recent years. The divisor function was one of them that had a lot of academic research done on it as it frequently occurs in the study of the asymptotic behavior of the Riemann zeta function.

The classical divisor function has a generalization to higher divisor functions. For an integer $\ell \geqslant2$, put
\begin{equation*} 
    d_\ell(n) = \sum_{\substack{m_1m_2...m_\ell = n \\  m_1,m_2,...,m_\ell \in \mathbb{Z}^+}} 1,
\end{equation*}
to be the $\ell$-th divisor function with the generating Dirichlet series $\zeta^\ell(s)$. When $\ell=2$, it represents the classical divisor function. The values of $d_\ell(n)$ fluctuate quite considerably as $n$ varies, but the average behavior is reasonably stable. Their statistical behavior can be measured in various ways, perhaps most interestingly by considering mean values over sparse sequences. One of the most commonly considered sequences is $\{P(n)\}_{n\geqslant 1}$ where $P(x)$ is a polynomial in $\mathbb{Z}[x]$. Hooley \cite{Hooley2} proved that 
\begin{equation*}
    \sum_{n \leqslant X^{1/2}} d(n^2 + a) = c_1 X^{1/2}\log{X} + c_2 X^{1/2} + O\left(X^{\frac{4}{9}}(\log{X})^3 \right),
\end{equation*}
for any fixed $a \in \mathbb{Z}$ such that $-a$ is not a perfect square, where $c_1$ and $c_2$ are constants depending only on $a$. For polynomials in more than one variable, the problem has been studied extensively for the von Mangoldt function $\Lambda(n)$. As an example, the proof of the infinite nature of primes of the form $x^2+y^4$ was achieved by Friedlander and Iwaniec \cite{FI06}, while the proof of the infinitude of primes of the form $x^3 + 2y^3$ was established by Heath-Brown \cite{HB}. Investigating the analogous sum with $\Lambda(n)$ replaced by the generalized divisor function $d_\ell(n)$ holds notable importance. Let $i\geqslant3$ be a positive integer. Many authors have considered the sums of the form
\begin{equation*} \label{gdf}
    \sum_{1 \leqslant n_1,n_2,...n_i \leqslant X^{1/2}} d_\ell(n_1^2 + n_2^2 + ...+ n_i^2).
\end{equation*}
For $\ell=2$ and $i=3$, there have been several results. Calder$\acute{o}$n and de Velasco in \cite{CV} provided an asymptotic formula 
\begin{equation} \label{cv} 
    \sum_{1 \leqslant n_1,n_2,n_3 \leqslant X^{1/2}} d(n_1^2 + n_2^2 + n_3^2) = \frac{4 \zeta(3)}{5 \zeta(5)} X^{3/2} \log{X} + O(X^{3/2}).
\end{equation}
In 2012, Guo and Zhai in \cite{GZ} improved \eqref{cv} and provided the second term with an error bound of $O(X^{4/3 + \varepsilon})$, where $\varepsilon$ is an arbitrary positive constant. Later, it was refined to $O(X \log^7{X})$ by L. L. Zhao in \cite{LZ}. The explorations of the error term turned out to be a significant issue in analytic number theory. In fact, if one takes the sum
\begin{equation*}
    \sum_{n\leqslant X} a_n d_\ell(n),
\end{equation*}
for some sparse sequence $a_n$, the situation becomes even more difficult for $\ell \geqslant3$. In addition to being of obvious interest in their
own, these sums arise as a tool in the study of the distribution of primes. Friedlander and Iwaniec in \cite{FI06} established
\begin{equation*}
    \sum_{\substack{n_1^2 + n_2^6 \leqslant X \\ (n_1,n_2)=1 }} d_3(n_1^2 + n_2^6) = c_3X^{2/3}(\log{X})^2 + O(X^{2/3}(\log{X})^{7/4} (\log\log{X})^{1/2}),
\end{equation*}
where $c_3$ is an effective constant. Sun and Zhang \cite{SunZhang} studied $d_3$ over ternary quadratic forms, and they proved
\begin{equation*}
    \sum_{1 \leqslant n_1,n_2,n_3 \leqslant X^{1/2}} d_3(n_1^2 + n_2^2 + n_3^2)  = c_4X^{3/2} (\log{X})^2 + c_5 X^{3/2} \log{X} + c_6 X^{3/2} + O_\varepsilon(X^{11/8+\varepsilon}),
\end{equation*}
for some constants $c_4,c_5$ and $c_6$. Let $k \geqslant3$ be an integer. Recently, Zhou and Hu \cite{ZHu} proved the following asymptotic formula:
\begin{equation*}
    \sum_{\substack{1 \leqslant n_1, n_2 \leqslant X^{1/2} \\ 1 \leqslant n_3 \leqslant X^{1/k}}} d_3(n_1^2 + n_2^2 + n_3^k) =  \textrm{main term of size } X^{1+\frac{1}{k}}+ O(X^{1 + \frac{1}{k}- \delta(k) + \varepsilon}),
\end{equation*} 
\begin{equation} \label{bound of d3}
    \delta(k) = \begin{cases}
        \frac{1}{15}, & k=3 \\
        \frac{1}{k 2^{k-1}}, & 4 \leqslant k \leqslant 7 \\
        \frac{1}{2k^2(k-1)}, & k \geqslant 8
    \end{cases}.
\end{equation}

One can ask such questions for an arithmetic function, which are Fourier coefficients of automorphic forms as well. Some progress has been made in understanding the Fourier coefficients of cusp forms for $\textrm{GL}(2)$. However, there is comparatively less information available regarding the Fourier coefficients for cusp forms of higher rank. Let $F \in S_{k_1}(N,\chi)$ be a holomorphic cusp form of weight $k_1$, level $N \in \mathbb{N}$, and character $\chi$. Let $\lambda_F(n)$ denote the normalized $n$-th Fourier coefficient of the form $F$. Blomer \cite{Blomer} established the following asymptotic result:
\begin{equation*}
    \sum_{n \leqslant X^{1/2}} \lambda_F(n^2+sn+t) = c_7X^{1/2} + O_{F,s,t,\varepsilon}(X^{3/7+\varepsilon}),
\end{equation*}
where $s,t \in \mathbb{Z}$, $F$ not necessarily an eigenform and $c_7=c_7(F, N)$ vanishes in many (but not all) cases. His result was improved by Templier and Tsimerman \cite{TT}. For polynomials in more variables, Acharya \cite{Acharya} proved 
\begin{equation*}
    \mathop{\sum\sum}_{\substack{n_1,n_2\in \mathbb{Z} \\ n_1^2 + n_2^2 \leqslant X}} \lambda_F(n_1^2 + n_2^2) \ll X^{1/2 + \varepsilon},
\end{equation*}
where $F \in S_{k_1}(4N,1)$.
Let $\phi$ be a Hecke-Maass cusp form of type $(\nu_1,\nu_2) \in \mathbb{C}^2$ for the group $\textrm{SL}(3,\mathbb{Z})$ with normalized Fourier coefficients $A(m,n)$ (i.e., $A(1,1)=1$). In a recent work \cite{HC}, we proved cancellation for the sum:
\begin{equation*}
    \mathop{\sum\sum}_{1 \leqslant n_1, n_2\leqslant X^{1/2}} A(1,Q(n_1,n_2)),
\end{equation*}
where $Q(x,y) \in \mathbb{Z}[x,y]$ is a homogeneous quadratic form. In an upcoming work, we generalized it and established the following result. For any $\varepsilon>0$,
\begin{equation*}
     \mathop{\sum\sum}_{1 \leqslant n_1, n_2\leqslant X^{1/2}} \mathcal{A}(Q(n_1,n_2)) = c_8 X(\log{X})^2 + c_9 X\log{X} + c_{10} X +O_\varepsilon(X^{1-\frac{1}{67} +\varepsilon}), 
\end{equation*}
for some constants $c_8, c_9$ and $c_{10}$. Here, $\mathcal{A}(n)$ is either $A(1,n)$ or $d_3(n)$ and $Q(x,y)\in \mathbb{Z}[x,y]$ is a binary quadratic polynomial, not necessarily a form.

As our primary result, we consider the Fourier coefficients of $\textrm{GL}(3)$ Hecke-Maass forms, namely $A(1,n)$ or $d_3(n)$ over polynomials $Q(x,y) + z^k$, for fixed integer $k\geqslant3$. For the divisor function, most have established their results using the classical circle method. We have used the DFI delta method (see Lemma \ref{dfi}) to prove our assertions. Let $X>1$ be a large number and $\mathsf{a}(n)$ be any bounded arithmetical function. Let
\begin{equation} \label{Sk(X)}
    \mathcal{S}_k(X) = \sum_{1 \leqslant Q(n_1, n_2) +n_3^k \leqslant X} \mathcal{A}(Q(n_1,n_2) + n_3^k)\mathsf{a}(n_3),
\end{equation}
where $k \geqslant3$ and $n_1, n_2, n_3 \in \mathbb{Z}^+$. Using Lemma \ref{ramanujan bound}, $d_3(n) \ll n^\varepsilon$ and the fact that $\mathsf{a}(n)$ is bounded. We get the following trivial bound for \eqref{Sk(X)}: 
\begin{equation*}
    \mathcal{S}_k(X) \ll X^{1+ \frac{1}{k} +\varepsilon}.
\end{equation*}
In our first theorem, we present the following results regarding the Fourier coefficients of special Hecke-Maass forms. We attain a stronger bound for each value of $ k > 3 $ in comparison to the bound described in equation \eqref{bound of d3}.
\begin{theorem} \label{thm1}
    Let $d_3(n)$ be the triple divisor function, and $\mathsf{a}(n)$ be either the von Mangoldt function $\Lambda(n)$ or the identity function. For any $\varepsilon>0$ and $Q(x,y) \in \mathbb{Z}[x,y]$ a binary quadratic polynomial, we have 
    \begin{align*} 
        \sum_{\substack{ n_1, n_2 \in \mathbb{Z} \\ 1 \leqslant n_3 \leqslant X^{1/k}}} d_3(Q(n_1,n_2) + n_3^k)&\Lambda(n_3)W_1\left(\frac{n_1}{X^{1/2}}\right)W_2\left(\frac{n_2}{X^{1/2}}\right) \\
        &\hspace{-3cm}= \frac{X^{1+\frac{1}{k}}}{4} (\log{X})^2 \mathcal{J}_0 \mathcal{S}_0 + \frac{X^{1+\frac{1}{k}}}{2}  (\log{X}) \left( \mathcal{J}_0\mathcal{S}_1 + \mathcal{J}_1\mathcal{S}_0\right) \\
        & + \frac{X^{1+\frac{1}{k}}}{2}\left( \mathcal{J}_0 \mathcal{S}_2+ \mathcal{J}_1\mathcal{S}_1 + \frac{1}{2}\mathcal{J}_2\mathcal{S}_0\right) +O\Big(X^{1+\frac{1}{k}-\delta(k)+\varepsilon}\Big),
\end{align*}
    and
\begin{align*} 
        \sum_{\substack{ n_1, n_2 \in \mathbb{Z} \\ 1 \leqslant n_3 \leqslant X^{1/k}}} d_3(Q(n_1,n_2) + n_3^k)&W_1\left(\frac{n_1}{X^{1/2}}\right)W_2\left(\frac{n_2}{X^{1/2}}\right) \\
        &\hspace{-3cm}= \frac{X^{1+\frac{1}{k}}}{4} (\log{X})^2 \mathcal{J}_0 \mathcal{C}_0 + \frac{X^{1+\frac{1}{k}}}{2}  (\log{X}) \left( \mathcal{J}_0\mathcal{C}_1 + \mathcal{J}_1\mathcal{C}_0\right) \\
        & \hspace{0.4cm}+ \frac{X^{1+\frac{1}{k}}}{2}\left( \mathcal{J}_0 \mathcal{C}_2+ \mathcal{J}_1\mathcal{C}_1 + \frac{1}{2}\mathcal{J}_2\mathcal{C}_0\right) +O\Big(X^{1+\frac{1}{k}-\delta(k)+\varepsilon}\Big),
\end{align*}
    where
    \begin{eqnarray*}
        \delta(k) = \begin{cases}
            \frac{1}{8} & \text{if } \ k=3\\
            \frac{1}{2k} & \text{if} \ \ k \geqslant 4
        \end{cases},
    \end{eqnarray*}
$W_1, W_2$ are smooth bump functions supported on the interval $[1,2]$ and have bounded derivatives. Here, for $j=0,1,2$, $\mathcal{J}_j$ are singular integrals defined in equation \eqref{J}, $\mathcal{S}_j$ and $\mathcal{C}_j$, are singular series defined in equations \eqref{S2}-\eqref{S0} and \eqref{C2main}-\eqref{C0main}, respectively. 
\end{theorem}
\begin{remark}
    As we are considering a weight over the third variable, we can restrict our sum to sequences with the weighted variable as prime powers. This type of sum has not been explored in earlier works.
\end{remark}
In our second theorem, we have proved cancellations for the general $\textrm{GL}(3)$ Fourier coefficients of Maass cusp forms.
\begin{theorem} \label{thm2}
    Let $A(1,n)$ represent the normalized $(1,n)$-th Fourier coefficients of Hecke-Maass cusp form $\phi$ on the group $\textrm{SL}(3,\mathbb{Z})$, and $\mathsf{a}(n)$ be any bounded arithmetic function. For any $\varepsilon>0$ and $Q(x,y) \in \mathbb{Z}[x,y]$ a binary quadratic polynomial, we have
    \begin{equation*} 
          \sum_{\substack{ n_1, n_2 \in \mathbb{Z} \\ 1 \leqslant n_3 \leqslant X^{1/k}}} A(1,Q(n_1,n_2) + n_3^k)\mathsf{a}(n_3)W_1\left(\frac{n_1}{X^{1/2}}\right)W_2\left(\frac{n_2}{X^{1/2}}\right) \ll \begin{cases}
              X^{\frac{7}{8} + \frac{1}{k} + \varepsilon} & \text{if} \ k =3 \\
              X^{1+\frac{1}{2k}+ \varepsilon}  & \text{if} \ k \geqslant4,
          \end{cases}
    \end{equation*}
where for $i=1,2, W_i$ are smooth bump functions supported on the interval $[1,2]$ and have bounded derivatives.    
\end{theorem}
\begin{remark}
     The proofs for both Theorem \ref{thm1} and Theorem \ref{thm2} exhibit deviations from the traditional method. We have not applied any summation formula to the variable of the smallest size. Furthermore, we have attained square root cancellation in the variable of the smallest size, which essentially provides the best possible bound. 
\end{remark}
\begin{remark}
    In the Theorem \ref{thm2}, $\mathsf{a}(n)$ is a general bounded arithmetical function. In particular, we can restrict our third variable over any thin sets, like square-free integers, primes, arithmetic progression with fixed moduli, Fourier coefficients of a holomorphic $\textrm{GL}(2)$ form, etc. 
\end{remark}
\begin{remark}
In the paper by Sun and Zhang \cite{SunZhang}, only the scenario with $ k=2 $ is examined. Zhou and Hu in \cite{ZHu} have established an asymptotic formula for $ k \geqslant 3$. Previous studies have primarily focused on diagonal forms. In our work, we extend the investigation to a broader scope of quadratic polynomials. In diagonal cases, evaluating the character sum is simplified as it results in Gauss sums. However, when considering more general polynomials, the character sum becomes more intricate. We have utilized $\ell$-adic techniques developed by Deligne and Katz to achieve the desired cancellations.
\end{remark} 
\begin{notation}
    Throughout the paper, the notation $a \ll A$ shall signify that, for any $\varepsilon> 0$ there exists a constant $c$ such that $|a| \leqslant cA X^\varepsilon$. The notation $ B \asymp C$, denotes that both $B \ll C$ and $C \ll B$. Additionally, when $D \sim E$ it implies that $E\leqslant D < 2E$. The symbol $\varepsilon$ represents a suitably small positive quantity that may vary at different instances, and $e(x) = e^{2\pi i x}$.
\end{notation}

\section{Preliminaries}
In this section, we will briefly outline some fundamental facts about the Hecke eigenforms of $\textrm{SL}(3,\mathbb{Z})$, the Voronoi summation formula for $\textrm{GL}(3)$ forms, the Voronoi summation formula for $d_3$, the Poisson summation formula, and other results utilized in our analysis.

\subsection{ \texorpdfstring{$\textrm{GL}(3)$}{} Voronoi Summation.}
Let $\phi$ be a Maass form of type $\nu = (\nu_1, \nu_2) \in \mathbb{C}^2$ for the group $\textrm{SL}(3,\mathbb{Z})$ such that $\phi$ is an eigenfunction of all the Hecke operators $T_n (n \in \mathbb{N})$ with Fourier coefficients $A(m,n) \in \mathbb{C}$, normalized so that $A(1,1) = 1$. \\

Let the Fourier-Whittaker expansion of $\phi$ be given by 
\begin{multline} \label{whittaker}
    \phi(z) = \sum_{\rho \in U_2(\mathbb{Z}) \symbol{92} \textrm{SL}(2,\mathbb{Z})} \sum_{m=1}^\infty \;\sum_{n \neq 0} \frac{A(m,n)}{|m n|} \\ \times \textbf{W}_{\text{Jacquet}} \left( \begin{pmatrix}
        |m n| & & \\ & m & \\ & & 1
    \end{pmatrix}, \begin{pmatrix}
        \rho & \\ & 1
    \end{pmatrix}z, \nu, \psi_{1, \frac{n}{|n|}}\right),
\end{multline}
where $\textbf{W}$ is the Whittaker-Jacquet function (for more details see Goldfeld's book \cite{Goldfeld}). We introduce the  Langlands parameters $(\alpha_1, \alpha_2, \alpha_3)$, which are defined by 
\begin{equation}\label{langland}
    \alpha_1= -\nu_1 -2\nu_2+1, \alpha_2 = -\nu_1+\nu_2  \text{ and } \alpha_3 = 2 \nu_1 + \nu_2 - 1.
\end{equation}
The Ramanujan-Selberg conjecture predicts that $|\text{Re}(\alpha_j)| = 0$ and based on the research of Jacquet-Shalika, we know that $| \text{Re}( \alpha_j)| < \frac{1}{2}$. The subsequent lemma provides the Ramanujan bound for $A(m,n)$ (refer to \cite{Molteni}).
\begin{lemma} \label{ramanujan bound}
Let $A(m,n)$ be as given in equation \eqref{whittaker}. Then
    \begin{equation}
        \mathop{\sum\sum}_{m^2 n \ll X} \; |A(m,n)|^2 \ll X^{1 + \varepsilon},
    \end{equation}
    where the implied constant depends on the form $\phi$ and $\varepsilon$.
\end{lemma}
We now recall the Voronoi summation formula for $\textrm{GL}(3)$ (refer \cite{MS}) and $d_3$ (refer \cite{XLI}), which will play a crucial role in our analysis. Let $g$ be a compactly supported function on $(0, \infty)$. The Mellin transform of g is defined by $\Tilde{g}(s) = \int_{0}^{\infty} g(x) x^{s-1} dx$, where $s = \sigma + i t$. For $\sigma > -1 + \text{max} \{ -\text{Re}(\alpha_1), - \text{Re}(\alpha_2), - \text{Re}(\alpha_3)\} $, and $\ell = 0, 1$, we define 
\begin{equation} \label{G}
    G_\ell(y) = \frac{1}{2\pi i} \int_{(\sigma)} (\pi^3 y)^{-s} \prod_{j=1}^3 \frac{\Gamma\left(\frac{1+s+\alpha_j + \ell}{2}\right)}{\Gamma\left(\frac{-s-\alpha_j + \ell}{2}\right)} \Tilde{g}(-s) \;ds,
\end{equation}
where $\alpha_j$ are Langlands parameters given in \eqref{langland}. We set
\begin{equation} \label{Gpm}
    G_{\pm}(y) = \frac{1}{2 \pi^{3/2}} \bigl(G_{0}(y) \mp i G_1(y)\bigr).
\end{equation}
The Kloosterman sum is defined by 
\begin{equation*}
    S(a,b;q) = \sideset{}{^\star}\sum_{x (\text{mod} \;q)} e\left(\frac{ax + b \overline{x}}{q} \right),
\end{equation*}
where $\overline{x}$ denotes multiplicative inverse of $x$ modulo $q$.
The Voronoi summation formula for $\textrm{GL}(3)$ forms is stated in the lemma that follows.
\begin{lemma} \label{voronoigl3}
    Let $g \in C_c^{\infty}(0, \infty)$. Let $A(m,n)$ be the $(m,n)$-th Fourier coefficients of a Maass form $\phi$ for $\textrm{SL}(3,\mathbb{Z})$, we have
    \begin{equation}
        \sum_{n=1}^{\infty} A(m,n) e\left(\frac{an}{q}\right) g(n) = q\sum_{\pm} \sum_{n_1 \mid qm} \sum_{n_2 = 1}^{\infty} \frac{A(n_1,n_2)}{n_1 n_2} S\left(m\overline{a},\pm n_2; \frac{mq}{n_1}\right) G_{\pm}\left(\frac{n_1^2n_2}{q^3 m}\right),
    \end{equation}
    where $(a,q) = 1$.
\end{lemma}
\subsection{ Voronoi summation for the triple divisor function \texorpdfstring{$d_3$}{}.}
The Voronoi formula for $d_3$ was proved by Ivi\'c, and Li \cite{XLI} later came up with a more explicit formula for it. To achieve his result, set 
\begin{equation} \label{sigma 00}
    \sigma_{0,0}(m,n) = \sum_{\substack{d_1\mid n \\ d_1 > 0}} \sum_{\substack{d_2 \mid \frac{n}{d_1}\\ d_2>0 \\ (d_2,m) = 1}} 1.
\end{equation}
For $h \in C_c^{\infty}(0,\infty),$ for $l=0,1$ and $\sigma > -1-2l$, set
\begin{equation*}
    H_l(y) = \frac{1}{2 \pi i} \int_{(\sigma)} (\pi^3 y)^{-s} \frac{\Gamma(\frac{1 + s +2l}{2})^3}{\Gamma(\frac{-s}{2})^3} \Tilde{h}(-s-l) ds,
\end{equation*}
and 
\begin{equation}
    H_\pm (y) = \frac{1}{2 \pi^{3/2}} \left(H_0(y) \mp \frac{i}{\pi^3 y} H_1(y)\right). 
\end{equation}
Observe that the behavior of $G_{\pm}$ is similar to $H_{\pm}$. Now with the aid of the above terminology, we state the Voronoi summation formula for $d_3$ in the following lemma.
\begin{lemma} \label{voronoi d3}
    Let $h \in C_c^{\infty}(0,\infty)$, $a, \overline{a}, q \in \mathbb{Z}^+$ with $a\overline{a} \equiv 1 \; (mod\; q )$. We have
\begin{dmath*}
\begin{aligned}
    \sum_{n \geqslant1}& d_3(n) e\left( \frac{a n}{q}\right) h(n) \\
   &= q \sum_{\pm}\sum_{n_1 \mid q} \sum_{n_2 =1}^\infty \frac{1}{n_1 n_2} \sum_{m_1 \mid n_1} \sum_{m_2 \mid \frac{n_1}{m_1}} \sigma_{0,0} \left( \frac{n_1}{m_1 m_2}, n_2 \right) S\left( \overline{a}, \pm n_2 ; \frac{q}{n_1}\right) H_\pm\left(\frac{n_1^2 n_2}{q^3}\right) \\
    & \hspace{1cm}+ \frac{1}{2q^2} \Tilde{h}(1) \sum_{n_1 \mid q} n_1 d(n_1)\; P_2(n_1,q)\; S\left( \overline{a}, 0 ; \frac{q}{n_1}\right) \\
    &\hspace{1cm}+ \frac{1}{2q^2} \Tilde{h}^\prime(1) \sum_{n_1 \mid q} n_1 d(n_1) \; P_1(n_1,q) \; S\left( \overline{a}, 0 ; \frac{q}{n_1}\right) \\
    &\hspace{1cm}+ \frac{1}{4q^2} \Tilde{h}\dblprime(1) \sum_{n_1 \mid q} n_1 d(n_1) \; S\left( \overline{a}, 0 ; \frac{q}{n_1}\right),
\end{aligned}
\end{dmath*}
where 
\begin{equation} \label{P1}
    P_1(n_1,q) = \frac{5}{3} \log{n} - 3 \log{q} + 3\gamma - \frac{1}{3 d(n_1)} \sum_{l \mid n_1} \log{l},
\end{equation}
and
\begin{multline} \label{P2}
    P_2(n_1,q) = (\log{n_1})^2 - 5\log{q} \log{n_1} + \frac{9}{2}(\log{q})^2 + 3\gamma^2 - 3\gamma_1 + 7\gamma \log{n_1} -9\gamma \log{q} \\
    + \frac{1}{d(n_1)} \left( (\log{n_1} + \log{q} - 5\gamma) \sum_{l \mid n_1} \log{l} -\frac{3}{2} \sum_{l \mid n_1}(\log{l})^2 \right),
\end{multline}
with $\gamma:= \lim_{s \rightarrow 1} \left(\zeta(s) - \frac{1}{s-1}\right)$ be the Euler constant and $\gamma_1:= -\frac{d}{ds} \left(\zeta(s) - \frac{1}{s-1}\right) \bigg|_{s=1}$ be the Stieltjes constant.
\end{lemma}
We can also have the Ramanujan bound in this case. By equation \eqref{sigma 00}, trivially, we have
\begin{equation} \label{ramanujan d3}
   \sum_{m \ll X} \Bigg| \sum_{m_1 \mid n} \sum_{m_2 \mid \frac{n}{m_1}} \sigma_{0,0}\left(\frac{n}{m_1 m_2}, m\right) \Bigg|^2 \ll \sum_{m \ll X} \left(d_3(n) d_3(m)\right)^2 \ll X^{1 + \varepsilon}.
\end{equation}
Now to apply the Voronoi summation formula in practice, one needs to know the asymptotic behavior of $G_0$ and $ G_1$. As $G_1$ has the same asymptotic behavior as $G_0$. All we need is the following lemma by X. Li \cite{XLI2}.
\begin{lemma} \label{GO}
 Suppose $g$ is a smooth function compactly supported on $[M, 2M]$. Let $G_0$ be as in \eqref{G}. Then for any fixed integer $k \geqslant1$ and $yM \gg 1$, we have
 \begin{equation*} 
     G_0(y) = \pi^4 y \int_{0}^\infty g(z) \sum_{j=1}^k \frac{c_j \cos{(6 \pi (yz)^{1/3})} + d_j \sin{(6 \pi (yz)^{1/3})} }{(\pi^3 yz)^{j/3}} dz + O\left((yM)^{\frac{-k + 2}{3}}\right),
 \end{equation*}
where $c_j$ and $d_j$ are constants depending on $\alpha_1, \alpha_2$ and $\alpha_3$. In particular, $c_1 =0$ and $d_1 = -\frac{2}{\sqrt{3\pi}}$.
\end{lemma}
\begin{proof}
    For the proof (see \cite[ Lemma $2.1$]{XLI2}).
\end{proof}

\subsection{ Poisson summation formula.} Let $\phi : \mathbb{R}^n \rightarrow \mathbb{C}$ be any Schwarz class function. The Fourier transform of $\phi$ is defined as 
\begin{equation*}
    \hat{\phi}(\mathbf{y}) = \mathop{\int ...\int}_{\mathbb{R}^n} \phi(\mathbf{x}) e(- \mathbf{x.y}) \; d\mathbf{x},
\end{equation*}
where $d\mathbf{x}$ is the usual Lebesgue measure on $\mathbb{R}^n$. In the following lemma, we state the Poisson summation formula.

\begin{lemma} \label{poisson}
    Let $\phi$ and $\hat{\phi}$ be as defined above. Then we have
    \begin{equation}
        \sum_{n \in \mathbb{Z}^n} \phi(n) = \sum_{m \in \mathbb{Z}^n} \hat{\phi}(m).
    \end{equation}
\end{lemma}
\begin{proof}
    Refer to \cite[page $69$, Theorem $4.4$]{IK}.
\end{proof}
\subsection{Gauss Sum} Let $Q(\bf{x})$ be a positive definite quadratic form in $\bf{x} = (x_1,x_2,...,x_r)$. We define the Gauss sum associated with a quadratic form $Q(\mathbf{x})$ as
\begin{equation} \label{Gauss sum}
    G_{\mathbf{m}}\left(\frac{a}{q}\right) := \sum_{\mathbf{x} \; \textrm{mod } q} e\left(\frac{a}{q} \left(Q(\mathbf{x}) + \mathbf{m}^t \mathbf{x}\right) \right).
\end{equation}
We have the following precise expression for the above-defined Gauss sum.
\begin{lemma} \label{gauss sum}
    Let $(q,2|\mathbf{A}| a) = 1$, where $\mathbf{A}$ is matrix associated with quadratic form $Q$ and $|\mathbf{A}|$ denotes the determinant of $\mathbf{A}$, and $\mathbf{m} \in \mathbb{Z}^r$. We have
    \begin{equation*} 
         G_{\mathbf{m}}\left(\frac{a}{q}\right) = \left(\frac{|\mathbf{A}|}{q}\right) \left( \varepsilon_q \left( \frac{2 a}{q}\right) \sqrt{q} \right)^r e\left( -\frac{a}{q} Q^*(\mathbf{m}) \right),
    \end{equation*}
where $Q^*(\mathbf{x})$ is adjoint quadratic form, $N$ is an integer such that $N Q^*(\mathbf{x})$ has integral coefficients and 
\begin{equation} \label{varepsilon}
   \varepsilon_q = \begin{cases}
 1 & \text{ if }  q \equiv 1 (\; mod \; 4) \\
i & \text{ if } q \equiv -1 (\; mod \; 4)
\end{cases}.
\end{equation}    
\end{lemma}
\begin{proof}
    See \cite[page $475$, Lemma $20.13$]{IK}.
\end{proof}

\noindent For a particular case, when $x \in \mathbb{Z}$ and $ m=0$ in equation \eqref{Gauss sum}, we have the following result.
\begin{lemma} \label{quadratic gauss sum}
If $(a,q) = 1$, we have
\begin{equation*}
    \sum_{x \; \textrm{mod } q} e\left( \frac{ax^2}{q}\right) = \begin{cases}
        0 & \text{if} \; q \equiv 2 \; \textrm{mod} \; 4 \\ \vspace{0.1cm}
        \varepsilon_q \sqrt{q} \left(\frac{a}{q} \right) & \text{if} \; q \equiv \pm 1 \; \textrm{mod} \; 4 \\ \vspace{0.1cm}
         \varepsilon_{q_0} \sqrt{q} \left( \frac{a}{q_0}\right) \lr{1 + e\lr{\frac{aq_0}{4}}} & \text{if} \; q =2^k q_0 \textrm{ with } k \equiv 0 \textrm{ mod } 2  \\ \vspace{0.1cm} 
        \varepsilon_{q_0} \sqrt{2q} \lr{\frac{a}{q_0}} e\lr{\frac{aq_0}{8}}  & \text{if} \; q =2^k q_0 \textrm{ with } k \equiv 1 \textrm{ mod } 2 
    \end{cases}
\end{equation*}
where $\varepsilon_q$ is defined in  \autoref{varepsilon}, $q_0$ is odd and $k\geq 2$.
\end{lemma}
\begin{proof}
    See (\cite{Khan+Young}, Lemma $5.1$).
\end{proof}
\subsection{Oscillatory integrals} We also recall the following estimates of exponential integrals. Let  $ g$ be a continuous function defined on an interval $[a, b] $.  Variation of the function $ g$ is defined by 
\begin{align} \label{variation of g}
\textrm{Var}_{[a,b]} g := \int_a^b |g^\prime(x)| dx .
\end{align}
Let
\[
I= \int_a^b g(x) e(F(x))dx. 
\]
\begin{lemma} \label{exponential sum}
{\bf (Exponential sum lemma)} Let $F$ be real and twice differentiable function on $[a,b]$, $g$ and $I$ are as above. Then if $F^\prime$ is monotone and $|F^\prime(x)|  \geqslant \mu_1 >0 $ for $a\leqslant x \leqslant b$, we have $I \ll \textrm{Var}_{[a,b]} g/\mu_1$. Further, if $|F\dblprime(x)| \geqslant\mu_2 >0$. Then we have $I \ll \textrm{Var}_{[a,b]} g/\mu_2^{1/2}$, where $\textrm{Var}_{[a,b]} g$ is defined in equation \eqref{variation of g}.

\end{lemma} 
\begin{proof}
See \cite[Lemma 2.1]{Hux}. 
\end{proof}
\section{DFI Method}
We will begin this part with a well-known Fourier expansion of the $\delta$-symbol, which was developed by Duke, Friedlander, and Iwaniec and is presented in Chapter $20$ of \cite{IK}. 
\begin{lemma} \label{dfi}
Let $\delta : \mathbb{Z} \rightarrow \{0,1\}$ be defined by
 \begin{align} 
  \delta (x)= 
  \begin{cases}
   1 \ \ \ \ \ \rm{if}\ \ \ \ \  x = 0, \\
   0 \ \ \ \ \ \rm{if} \ \ \ \ \ x \neq 0.
  \end{cases}
 \end{align}
Then for $n,m \in \mathbb{Z} \cap [-2L,2L],$ we have
 \begin{equation} \label{delta}
     \delta(n,m) = \delta(n-m) = \frac{1}{\mathcal{Q}} \sum_{q =1}^\infty \frac{1}{q}\;
     \sideset{}{^\star}\sum_{a \;mod \;q} e\left(\frac{(n-m)a}{q}\right) \int_{\mathbb{R}} \psi(q,x)e\left(\frac{(n-m)x}{q\mathcal{Q}}\right)dx,
 \end{equation}
where $\mathcal{Q}=2 L^{1/2}$. The function $\psi$ in \eqref{delta} is not explicitly given. Nevertheless, the following properties of the function $\psi(q,x)$ are of our interest:
 \begin{align}
     &\psi(q,x) = 1 + h(q,x), \;\;\; \text{with} \;\;\; h(q,x) = O\left(\frac{\mathcal{Q}}{q} \left(\frac{q}{\mathcal{Q}} + |x| \right)^A \right), \label{delta1}\\
     &\psi(q,x) \ll |x|^{-A}, \label{delta2} \\
     &x^j\frac{\partial^j}{\partial x^j} \psi(q,x) \ll \min \left\{ \frac{\mathcal{Q}}{q}, \frac{1}{|x|}\right\} \log \mathcal{Q}, \;\; \text{for any}\;\; A>1, \;j \geqslant1. \label{delta3}
 \end{align}
 In particular, \eqref{delta2} implies that the effective range of integral in \eqref{delta} is $[-L^\epsilon, L^\epsilon].$ It also follows from \eqref{delta1} that if $q \ll \mathcal{Q}^{1-\epsilon}$ and $x \ll \mathcal{Q}^{-\epsilon}$, then $\psi(q,x)$ can be replaced by 1, at the cost of a negligible error term. If $q \gg \mathcal{Q}^{1-\epsilon}$, then we get $x^j\frac{\partial^j}{\partial x^j} \psi(q,x) \ll \mathcal{Q}^{\epsilon},$ for any $j \geqslant1$. If $q \ll \mathcal{Q}^{1- \epsilon}$ and $\mathcal{Q}^{-\epsilon} \ll |x| \ll \mathcal{Q}^{\epsilon}$, then $x^j\frac{\partial^j}{\partial x^j} \psi(q,x) \ll \mathcal{Q}^{\epsilon},$ for any $j \geqslant1$. Finally, by Parseval and Cauchy, we get
  \begin{equation*}
      \int_{\mathbb{R}} |\psi(q,x)| \; + \; |\psi(q,x)|^2 \; dx \ll \mathcal{Q}^\varepsilon,
  \end{equation*}
  i.e., $\psi(q,x)$ has average size one in the $L^1$ and $L^2$ sense.
 \end{lemma}
 \begin{proof}
For the proof, we refer to Chapter $20$ equation $(20.157)$ of the book \cite{IK} by Iwaniec and Kowalski.
 \end{proof}
Hence, we can view $\psi(q,x)$ as a nice weight function. 
\section{Idea behind the proof}
In this section, we will discuss the method and main ideas to get non-trivial cancellations in the sum $\mathcal{S}_k(X)$. We apply the DFI method given in Lemma \ref{dfi} with $\mathcal{Q} = X^{1/2}$. For simplicity, let us consider the generic case, i.e., $q \sim \mathcal{Q}$. The expression for $\mathcal{S}_k(X)$ becomes
\begin{equation*}
      \sum_{q\sim \mathcal{Q}} \ \sideset{}{^\star}\sum_{a \ \textrm{mod } q}  \sum_{r \sim X} \mathcal{A}(r)e\left(\frac{ar}{q}\right) \sum_{n_1, n_2 \sim X^{1/2}} \sum_{n_3\leqslant X^{1/k}} \mathsf{a}(n_3) e\left(\frac{-a(Q(n_1,n_2) + n_3^k)}{q}\right).
\end{equation*}
Next, we apply the $\textrm{GL}(3)$ Voronoi to the $r$-sum (refer to subsection \ref{subsection voronoi} for details), the dual length and savings become 
\begin{equation*}
     r^* \sim \frac{\mathcal{Q}^3}{X} \ \text{and savings} \ = \frac{X}{\mathcal{Q}^{3/2}}.
\end{equation*}
After that, apply the Poisson summation formula to the sums over $n_1$ and $n_2$, the dual length and savings become
\begin{equation*}
     n_1^*,n_2^* \sim \frac{\mathcal{Q}}{X^{1/2}} \ \text{and savings} \ =\frac{X^{1/2}}{\mathcal{Q}^{1/2}}.
\end{equation*}
The details of which are captured in subsection \ref{subsection poisson}. After the summation formulae, we get the following character sum over $a$:
\begin{equation*}
    \mathfrak{S}(...) =\sideset{}{^\star}\sum_{a \ \textrm{mod } q} S(\overline{a}, r^*;q) e\left( \frac{-\overline{4a}Q^*(n_1^{*}, n_2^{*})}{q}\right)e\left(\frac{-an_3^k}{q}\right),
\end{equation*}
for some polynomial $Q^*$. We save $\sqrt{\mathcal{Q}}$ from the $a$-sum (for details see subsection \ref{character sum}). Hence, in total we have saved
\begin{equation*}
    \frac{X}{\mathcal{Q}^{3/2}} \times \frac{X^{1/2}}{\mathcal{Q}^{1/2}} \times \frac{X^{1/2}}{\mathcal{Q}^{1/2}} \times \sqrt{\mathcal{Q}} = X.
\end{equation*}
Now, we are on the boundary, we need to save a little extra to get the non-trivial cancellations in $\mathcal{S}_k(X)$. To achieve that, we apply Cauchy's inequality to the $r^*$-sum in the following expression:
\begin{equation*}
    \sum_{q \sim \mathcal{Q}} \mathop{\sum\sum}_{n_1^*,n_2^* \sim 1} \sum_{r^* \sim X^{1/2}} \frac{|B(1,r^*)|}{r^{* 1/3}} \left| \sum_{n_3 \leqslant X^{1/k}}\ \mathsf{a}(n_3) \mathfrak{S}(...)\right|.
\end{equation*}
We have suppressed the presence of integral transforms as it has no oscillations. After Cauchy, we arrive at
\begin{equation*}
    X^{1/4} \times \sum_{q \sim \mathcal{Q}} \ \mathop{\sum\sum}_{n_1^*,n_2^* \sim 1} \ \left(\sum_{r^* \sim X^{1/2}} \Big| \sum_{n_3 \leqslant X^{1/k}}\ \mathsf{a}(n_3) \mathfrak{S}(...)\Big|^2 \right)^{1/2}.
\end{equation*}
Again, we apply the Poisson summation formula to the $r^*$-sum. In the zero frequency case ($r^* =0$), we save $X^{1/2k}$. In the non-zero frequencies ($r^*\neq 0$), we save $(X^{1/2}/\mathcal{Q}^{1/2})^{1/2} = X^{1/8}$. Consequently, we save $\min\{X^{1/2k}, X^{1/8}\}$ over the trivial bound. The detailed proof for both the theorems are presented in the rest of the paper.
\section{Proof of Theorem \ref{thm1} and Theorem \ref{thm2}} \label{section 4}
For a fixed $k \geqslant3$ and $Y = X^{1/k}$. We have the following sum from equation \eqref{Sk(X)}:
\begin{equation*}
	\mathcal{S}_k(X) =\sum_{ n_1, n_2 \in \mathbb{Z}} \sum_{n_3\leqslant Y} \mathcal{A}(Q(n_1,n_2) + n_3^k) \mathsf{a}(n_3)W_1 \left(\frac{n_1}{X^{1/2}}\right)W_2\left(\frac{n_2}{X^{1/2}}\right),
\end{equation*}
where $W_1$ and $W_2$ are smooth bump functions supported on the interval $[1,2]$ and have bounded derivatives, i.e., 
\begin{equation*}
	x^j W_l^{(j)}(x) \ll_j 1, \;\; \text{for}\;\; l=1,2\;\; \text{and for all non-negative integers}\;\; j. 
\end{equation*}
We apply the expansion of $\delta$-symbol as given in Lemma \ref{dfi}, equation $\eqref{delta}$. We can rewrite $\mathcal{S}_k(X)$ as 
\begin{equation*}
	\mathcal{S}_k(X) = \sum_{n_1,n_2 \in \mathbb{Z}} \sum_{n_3\leqslant Y} \sum_{r \in \mathbb{Z}} \mathcal{A}(r) \mathsf{a}(n_3)\delta(r,Q(n_1,n_2) + n_3^k) V\left( \frac{r}{X} \right) W_1 \left(\frac{n_1}{X^{1/2}}\right)W_2\left(\frac{n_2}{X^{1/2}}\right),
\end{equation*}
where $V$ is a smooth bump function supported on the interval $[1/2,3]$, $V(x) \equiv 1$ on the interval $[1,2]$ and $x^j V^{(j)}(x) \ll_j 1$  for any $j \geqslant0$. We obtain
\begin{multline} \label{Sk(X) before summation}
	\mathcal{S}_k(X) = \frac{1}{\mathcal{Q}}  \sum_{1 \leqslant q \leqslant \mathcal{Q}} \int_{\mathbb{R}} \frac{\psi(q, u)}{q} U(u)\sideset{}{^\star}\sum_{a(q)} \left[ \sum_{r \in \mathbb{Z}} \mathcal{A}(r)e\left(\frac{ar}{q}\right)V\left(\frac{r}{X}\right) e\left(\frac{ru}{q\mathcal{Q}}\right) \right] \\ \times \sum_{n_1, n_2 \in \mathbb{Z}} \sum_{n_3\leqslant Y} \mathsf{a}(n_3) e\left(\frac{-a(Q(n_1,n_2) + n_3^k)}{q}\right)  W_1 \left(\frac{n_1}{X^{1/2}}\right)W_2\left(\frac{n_2}{X^{1/2}}\right)\\ \times \; e\left(\frac{-u(Q(n_1,n_2) + n_3^k)}{q\mathcal{Q}}\right) du,  
\end{multline}
where $U$ is also a smooth bump function supported on $[-2X^\varepsilon, 2X^\varepsilon]$ with $U(x) \equiv 1$ on $[-X^\varepsilon, X^\varepsilon]$ and $x^j U^{(j)}(x) \ll_j 1$, for any $j \geqslant0$. In the following subsections, we shall now apply summation formulae.


\subsection{Voronoi summation formula.} \label{subsection voronoi}
We analyze the sum over $r$ using the $\textrm{GL}(3)$ Voronoi summation formula given in Lemma \ref{voronoigl3} (for Fourier coeffcients of $\textrm{SL}(3,\mathbb{Z})$ forms) or, alternatively, in Lemma \ref{voronoi d3} (for $d_3$) with $v_u(y) = V\left(\frac{y}{X}\right)e\left(\frac{yu}{q\mathcal{Q}}\right)$.
\begin{multline} \label{voronoi1}
    \sum_{r \in \mathbb{Z}} \mathcal{A}(r)e\left(\frac{ar}{q}\right)v_u(r) =  q \sum_{\pm} \sum_{n\mid q} \sum_{m = 1}^\infty \frac{B(n,m)}{nm} S\left(\overline{a}, \pm m; \frac{q}{n}\right) G_{\pm}\left(\frac{n^2 m}{q^3}\right) \\ + \frac{1}{2q^2} \Tilde{\mathcal{V}}(v_u,a,q),
\end{multline}
where
\begin{equation} \label{coeff B(n,m)}
  B(n,m) = \Biggl\{ \begin{array}{ll}
    A(n,m) & \quad \text{if } \mathcal{A}(r) = A(1,r) \\ \vspace{0.2cm}
   \sum_{m_1 \mid n}\sum_{m_2 \mid \frac{n}{m_1}} \sigma_{0,0} \left(\frac{n}{m_1m_2}, m\right) & \quad \text{if } \mathcal{A}(r) = d_3(r) 
      
 \end{array},\Biggr.
\end{equation}
and
\begin{multline} \label{V tilde}
    \Tilde{\mathcal{V}}(v_u,a,q) = \Tilde{v_u}(1) \sum_{n \mid q} n d(n) P_2(n,q)\; S\left( \overline{a}, 0 ; \frac{q}{n}\right) \\
    + \Tilde{v_u}^\prime(1) \sum_{n \mid q} n d(n) P_1(n,q) S\left( \overline{a}, 0 ; \frac{q}{n}\right) \\
    + \frac{1}{2} \Tilde{v_u}\dblprime(1) \sum_{n \mid q} n d(n) S\left( \overline{a}, 0 ; \frac{q}{n}\right).
\end{multline}
The contribution of the term $\Tilde{\mathcal{V}}(v_u,a,q)$ to $\mathcal{S}_k(X)$ will give us the main term, which we will study in the subsequent section.
Next, we extract the oscillations of $G_\pm$ using Lemma \ref{GO} and equation \eqref{Gpm}. For $yX \gg X^{\varepsilon}$, we have
\begin{equation*}
    G_\pm(y) = \pi^4 y \int_{0}^\infty v_u(z) \sum_{j=1}^{k_0} \frac{c_je{(3 (y z)^{1/3})} + d_je(-3(yz)^{1/3})}{(\pi^3 yz)^{j/3}} dz + O(X^{-2024}),
\end{equation*}
where $k_0 = \floor{\frac{6072}{\varepsilon} +2} +1$ and $\floor{.}$ denotes the greatest integer function. Analysis for the complementary range is done in the subsection \ref{error}. Substituting the value of $v_u$ and considering the term corresponding to $j=1$, we get
\begin{multline*}
    G_\pm(y) =  \frac{-2 \pi^3 y^{2/3}}{\sqrt{3\pi}} \int_{0}^\infty V\left(\frac{z}{X}\right) e\left(\frac{zu}{q\mathcal{Q}}\right) z^{-1/3} \left(e(3 (y z)^{1/3}) - e(-3(yz)^{1/3})\right) dz \\ + O(X^{-2024}),
\end{multline*}
as the other terms can be treated similarly and in fact, give us better estimates. Using the change of variable $z \rightarrow X z$ and putting $y = n^2 m / q^3$, up to some lower order terms, we arrive at
\begin{align*}
    G_\pm\left(\frac{n^2 m}{q^3}\right)
    &= \frac{X^{2/3}}{q^2} (n^2 m)^{2/3} \mathcal{I}_\pm(n^2 m, u, q),  \ \ \ \textrm{where}
\end{align*}

\begin{equation} \label{I integral}
    \mathcal{I}_\pm(n^2 m, u,q) =  \frac{-2\pi^3 }{ \sqrt{3 \pi}} \int_0^\infty V^\pm(z) z^{-1/3} e\left( \frac{X z u}{q \mathcal{Q}} \pm \frac{3(X z n^2 m)^{1/3}}{q}\right) dz,
\end{equation}
and $V^\pm(z) = \pm V(z)$.
Using integration by parts repeatedly, the integral $\mathcal{I}_\pm(n^2 m, u,q)$ is negligibly small if 
\begin{equation} \label{K}
    n^2 m \gg \frac{q^3}{X} + X^{1/2} u^3=: K.
\end{equation}
Thus, from equation \eqref{voronoi1} we have 
\begin{align} \label{voronoi yx>}
    \sum_{r \in \mathbb{Z}} \mathcal{A}(r) e\left(\frac{ar}{q}\right) v_u(r) = \frac{X^{2/3}}{q} \sum_{\pm} \sum_{n \mid q} \sum_{m=1}^\infty \frac{B(n,m)}{n^{-1/3}m^{1/3}} S\left(\overline{a}, \pm m; \frac{q}{n}\right) \mathcal{I}_{\pm} (n^2m, u,q) \notag \\ + \frac{1}{2q^2} \Tilde{\mathcal{V}}(v_u,a,q) + O(X^{-2024}),
\end{align}
where the sum over $m$ is negligibly small unless $n^2 m \ll K$, and the second term on the right side will appear only in the case of $d_3$.

\subsection{Poisson summation formula.} \label{subsection poisson} We first write $n_1$ and $n_2$ as $n_1 = \alpha_1 + \ell_1 q, n_2 = \alpha_2 + \ell_2 q$ then we apply the Poisson summation formula given in Lemma \ref{poisson} to the sums over $\ell_1$ and $\ell_2$ as follows
 \begin{align} \label{T sum}
\hspace{-1.5 cm} \mathcal{T}:= \sum_{n_1, n_2 \in \mathbb{Z}} e\left(\frac{-a Q(n_1,n_2)}{q}\right) W_1 \left(\frac{n_1}{X^{1/2}}\right)W_2\left(\frac{n_2}{X^{1/2}}\right) e\left(\frac{-uQ(n_1,n_2)}{q\mathcal{Q}}\right)
\end{align}
\begin{multline*}	
\hspace{0.4cm} = \sum_{\alpha_1, \alpha_2 (mod \;\;q)} e\left(\frac{-aQ(\alpha_1, \alpha_2)}{q}\right) \sum_{\ell_1, \ell_2 \in \mathbb{Z}}  W_1 \left(\frac{\alpha_1 + \ell_1 q}{X^{1/2}}\right)  W_2\left(\frac{\alpha_2 + \ell_2 q}{X^{1/2}}\right) \\  \times e\left( \frac{-uQ(\alpha_1+\ell_1 q, \alpha_2+\ell_2 q)}{q\mathcal{Q}} \right)
\end{multline*} 
\begin{align*}
 = \sum_{\alpha_1, \alpha_2(mod \;\;q)} e\left(\frac{-aQ(\alpha_1, \alpha_2)}{q}\right) \sum_{m_1, m_2 \in \mathbb{Z}} \ \iint_{\mathbb{R}^2} W_1 \left(\frac{\alpha_1 + x q}{X^{1/2}}\right) W_2\left(\frac{\alpha_2 + y q}{X^{1/2}}\right) \\ \times e\left( \frac{-uQ(\alpha_1+x q, \alpha_2+y q)}{q\mathcal{Q}} \right)    e(-m_1x - m_2y) \;dx dy.
\end{align*}
Substituting the change of variables $(\alpha_1 + x q)/X^{1/2} = v_1$ and $(\alpha_2 + y q)/X^{1/2} = v_2$, we obtain
\begin{equation} \label{T poisson}
\mathcal{T} = \frac{X}{q^2} \sum_{m_1, m_2 \in \mathbb{Z}} \mathfrak{C}(m_1,m_2,a,q) \;\;\mathfrak{J}(m_1,m_2,u,q), 
\end{equation}
where the character sum $\mathfrak{C}(m_1,m_2,a,q)$ is given by
\begin{equation} \label{char poisson}
    \mathfrak{C}(m_1,m_2,a;q) = \sum_{\alpha_1, \;\alpha_2( mod \;q)} e\left(\frac{-aQ(\alpha_1, \alpha_2) + m_1\alpha_1 + m_2\alpha_2}{q}\right), 
\end{equation}
and the integral transform $\mathfrak{J}(m_1,m_2,r,q)$ is given by 
\begin{multline} \label{integral poisson}
    \mathfrak{J}(m_1,m_2,u,q)= \iint_{\mathbb{R}^2} W_1(v_1) W_2(v_2) e\left( \frac{-m_1 v_1 X^{1/2} - m_2 v_2X^{1/2}}{q}\right) 
    \\ \times e\left(\frac{-uQ(v_1 X^{1/2},v_2 X^{1/2})}{q\mathcal{Q}} \right) dv_1dv_2. 
\end{multline}
By applying integration by parts $j$ times and using 
\begin{equation*}
    \frac{\partial^j}{\partial v_1^j} e\left(\frac{-uQ(v_1 X^{1/2},v_2 X^{1/2})}{q \mathcal{Q}} \right) \ll \left(\frac{X}{q\mathcal{Q}}\right)^j.
\end{equation*}
We obtain that, for $m_1, m_2 \neq 0$,  the sum over $m_1$ and $m_2$ is negligibly small unless $m_1\ll X^\varepsilon $ and $m_2 \ll X^\varepsilon$ as $\mathcal{Q}= \sqrt{X}$.
Thus, from the above analysis, we can write the equation \eqref{Sk(X) before summation} as
\begin{equation} \label{split in main + error}
    \mathcal{S}_k(X) = \mathcal{S}_{k,M}(X) + \mathcal{S}_{k,E}(X),
\end{equation}
where
\begin{multline} \label{main term}
    \mathcal{S}_{k,M}(X) = \frac{X}{2\mathcal{Q}} \sum_{1\leqslant q \leqslant \mathcal{Q}} \frac{1}{q^5} \sum_{n_3 \leqslant Y} \mathsf{a}(n_3)  \\ \times \mathop{\sum\sum}_{m_1,m_2 \in \mathbb{Z}} \int_{\mathbb{R}} \psi(q,u) U(u) \mathfrak{C}_1(...) \mathfrak{J}(m_1,m_2,u,q) e\left(\frac{-un_3^k}{q\mathcal{Q}}\right) du,
\end{multline}
and
\begin{multline} \label{error term}
    \mathcal{S}_{k,E}(X) = \frac{X^{5/3}}{\mathcal{Q}} \sum_{1 \leqslant q \leqslant \mathcal{Q}} \frac{1}{q^4} \sum_{\pm} \sum_{n \mid q} \sum_{n^2 m \ll K} \frac{B(n,m)}{n^{-1/3} m^{1/3}} \sum_{n_3 \leqslant Y} \mathsf{a}(n_3) \\ \times \mathop{\sum\sum}_{m_1,m_2 \in \mathbb{Z}} \mathfrak{C}_2(...) \int_{\mathbb{R}} \psi(q,u) U(u) \mathcal{I}_{\pm}(n^2m,u,q) \mathfrak{J}(m_1,m_2,u,q) e\left(\frac{-u n_3^k}{q \mathcal{Q}} \right) du.
\end{multline}
Here, the character sums are given by
\begin{equation} \label{C1}
    \mathfrak{C}_1(m_1,m_2,n_3,a;q) = \sideset{}{^\star}\sum_{a \; \textrm{mod } q} \Tilde{V}(v_u,a,q) \mathfrak{C}(m_1,m_2,a;q) e\left(\frac{-a n_3^k}{q}\right),  \ \ \textrm{and }
\end{equation}
\begin{equation} \label{C2}
    \mathfrak{C}_2(m_1,m_2,m,n,n_3;q) = \; \sideset{}{^\star}\sum_{a \; \textrm{mod } q} S\left(\overline{a}, \pm m; \frac{q}{n}\right) \; \mathfrak{C}(m_1,m_2,a;q) e\left(\frac{-a n_3^k}{q}\right).
\end{equation}  
Note that the sum $\mathcal{S}_{k,M}(X)$ given in equation \eqref{main term} will appear only in the case of $d_3$. 
\subsection{Simplification of integrals.} \label{subsection integral simplify}
Let
\begin{equation} \label{Lpm}
    \mathcal{L}^{\pm}(m_1,m_2,n,m,n_3,q) := \int_{\mathbb{R}} \psi(q,u) U(u) \mathcal{I}_{\pm}(n^2m,u,q) \mathfrak{J}(m_1,m_2,u,q) e\left(\frac{-u n_3^k}{q \mathcal{Q}} \right) du.
\end{equation}
Putting the expressions for $\mathcal{I}_{\pm}$ and $\mathfrak{J}$ from equations \eqref{I integral} and \eqref{integral poisson}, respectively. We arrive at
\begin{multline*}
     \mathcal{L}^{\pm}(...) = \frac{-2\pi^{5/2}}{\sqrt{3}} \int_{\mathbb{R}} \psi(q,u) U(u) \int_{0}^{\infty} V^{\pm}(z) z^{-1/3} e\Biggl( \frac{X zu}{q \mathcal{Q}} \pm \frac{3(X z n^2 m)^{1/3}}{q}\Biggr) dz \\  
      \hspace{20pt} \times \iint_{\mathbb{R}^2} W_1(v_1) W_2(v_2) e\left( \frac{-m_1 v_1 X^{1/2} - m_2 v_2X^{1/2}}{q}\right) e\left(\frac{-uQ(v_1 X^{1/2},v_2 X^{1/2})}{q\mathcal{Q}} \right) \\ \times e\left(\frac{-u n_3^k}{q \mathcal{Q}} \right)  dv_1dv_2 du. 
\end{multline*}
Now, consider the $u$-integral
\begin{equation*}
    \int_{\mathbb{R}} \psi(q,u) U(u) e\Biggl( \frac{(Xz - Q(v_1 X^{1/2},v_2 X^{1/2}) -n_3^k)u }{q\mathcal{Q}} \Biggr) du,
\end{equation*}
for small $q$, i.e., $q \ll \mathcal{Q}^{1-\varepsilon}$, we split the $u$-integral into two parts $| u | \ll \mathcal{Q}^{-\varepsilon}$ and $| u | \gg \mathcal{Q}^{-\varepsilon}$.  If $|u| \ll \mathcal{Q}^{-\varepsilon}$ then from Lemma \ref{dfi} we can replace $\psi(q,u)$ by $1$, up to some negligible error, we get
\begin{equation*}
    \int_{|u| \ll \mathcal{Q}^{-\varepsilon}} U(u) e\left( \frac{( Xz - Q(v_1 X^{1/2},v_2 X^{1/2}) -n_3^k)u }{q \mathcal{Q}}\right) du.
\end{equation*}
Applying integration by parts repeatedly, we get that the integral is negligible unless $q \mathcal{Q} \gg   |zX - Q(v_1 X^{1/2},v_2 X^{1/2}) - n_3^k|$ or $|z -Q(v_1 X^{1/2},v_2 X^{1/2})/X - n_3^k/X| \ll q\mathcal{Q}/X$. If $|u| \gg \mathcal{Q}^{-\varepsilon}$, we have 
\begin{equation*}
     \int_{|u|\gg \mathcal{Q}^{-\varepsilon} }\psi(q,u) U(u) e\Biggl( \frac{( zX - Q(v_1 X^{1/2},v_2 X^{1/2}) -n_3^k)u}{q\mathcal{Q}} \Biggr) du,
\end{equation*}
again applying integration by parts repeatedly, using the properties of involved weight functions given in Lemma \ref{dfi} equation \eqref{delta3} and $x^j U^{(j)}(u) \ll_j 1,$ we have  
\begin{equation}\label{bump function U}
   \frac{\partial^j}{\partial u^j}\psi(q,u) \ll_j \mathcal{Q}^{\varepsilon j} \ \ \text{and} \ \ U^{j}(u) \ll_j \mathcal{Q}^{\varepsilon j}, \hspace{1cm} \text{for any} \ j\geqslant1.
\end{equation}
In this case, also, the integral is negligibly small unless $|z -Q(v_1 X^{1/2},v_2 X^{1/2})/X - n_3^k/X| \ll \frac{q\mathcal{Q}}{X}$. For large $q$, i.e., if $q \gg \mathcal{Q}^{1-\varepsilon}$, this condition holds trivially.\\

 Let $z -Q(v_1 X^{1/2},v_2 X^{1/2})/X - n_3^k/X =: t$ with $|t| \ll \frac{q\mathcal{Q}}{X}$, we arrive at the following expression for $\mathcal{L}^{\pm}$, upto some negligible error term.
\begin{align} \label{V1}
    \mathcal{L}^{\pm}(...) &= \frac{-2\pi^{5/2}}{\sqrt{3}} \int_{|t| \ll \frac{q\mathcal{Q}}{X}} \iint_{\mathbb{R}^2} V_1^{\pm}\left(\frac{Xt + Q(v_1 X^{1/2},v_2 X^{1/2}) + n_3^k }{X}\right) \notag\\  
     &\times \; W_1(v_1) W_2(v_2) e\Biggl(\pm \frac{3( (Xt + Q(v_1 X^{1/2},v_2 X^{1/2}) + n_3^k ) n^2 m)^{1/3}}{q}\Biggr) \notag \\  &\times \; e\Biggl( \frac{-m_1 v_1 X^{1/2} - m_2 v_2X^{1/2}}{q} \Biggr) dtdv_1 dv_2, 
\end{align}
with $V_{1}^{\pm}(z) = V^{\pm}(z) z^{-1/3}$, which is also a smooth compactly supported function and $V_{1}^{\pm(j)}(z) \ll_j 1$ for $j\geqslant0$. 
Executing the remaining integrals trivially, we get 
\begin{equation*}
    \mathcal{L}^{\pm}(m_1,m_2,n,m,n_3,q) \ll \frac{q\mathcal{Q}}{X} = \frac{q}{\mathcal{Q}}.
\end{equation*}
This proves the following lemma.
\begin{lemma} \label{Lintegral}
Let $ \mathcal{L}^{\pm}(m_1,m_2,n,m,n_3,q)$ be as given in equation \eqref{Lpm}. Then we have
    \begin{equation*}
         \mathcal{L}^{\pm}(m_1,m_2,n,m,n_3,q) \ll \frac{q}{\mathcal{Q} }.
    \end{equation*}
\end{lemma}
In the following subsections, we will analyze $\mathcal{S}_{k,E}(X)$. We aim to achieve non-trivial cancellations in it. 


     
\subsection{Applying Cauchy-Schwarz inequality.} \label{4.4 cauchy} In this subsection, we apply the Cauchy-Schwartz inequality to the $m$-sum in equation \eqref{error term} so that we can get rid of the coefficients $B(n,m)$.
\begin{multline} \label{sum S}
    \mathcal{S}_{k,E}(X) = \frac{X^{5/3}}{\mathcal{Q}} \sum_{1 \leqslant q \leqslant \mathcal{Q}} \frac{1}{q^4}   \mathop{\sum\sum}_{m_1,m_2 \in \mathbb{Z}} \sum_{n \mid q} \sum_{n^2 m \ll K} \frac{B(n,m)}{n^{-1/3} m^{1/3}} \sum_{n_3 \leqslant Y } \mathsf{a}(n_3) \\ \times \mathfrak{C}_2(...) \mathcal{L}^{\pm}(...).
\end{multline}
Splitting the sum over $q$ into dyadic blocks $q \sim D \ll \mathcal{Q}$, we see that $\mathcal{S}_{k,E}(X)$ is bounded by 
\begin{equation*}
      \sup_{D \ll \mathcal{Q}} \frac{X^{5/3}}{\mathcal{Q} D^4}  \sum_{q \sim D} \sum_{\substack{n \ll D \\ n \mid q}} n^{1/3}  \mathop{\sum\sum}_{m_1,m_2 \in \mathbb{Z}} \sum_{m \ll K/n^2} \frac{|B(n,m)|}{m^{1/3}}   \left| \sum_{n_3 \leqslant Y} \mathsf{a}(n_3)\mathfrak{C}_2(...) \mathcal{L}^{\pm}(...)  \right|.
\end{equation*}
On applying Cauchy's inequality to the $m$-sum, we arrive at
\begin{equation} \label{final Sk(X)}
    \mathcal{S}_{k,E}(X) \ll \sup_{D \ll \mathcal{Q}} \frac{X^{5/3}}{\mathcal{Q}D^4} \sum_{q \sim D} \sum_{\substack{n \ll D \\ n \mid q}} n^{1/3} \mathop{\sum\sum}_{m_1,m_2 \in \mathbb{Z}} \Theta^{1/2}\; \Omega^{1/2},
\end{equation}
where 
\begin{equation}
    \Theta = \sum_{m \ll K/n^2} \frac{|B(n,m)|^2}{m^{2/3}},
\end{equation}
and by incorporating a smooth bump function, say $W_3$, we smooth out the sum over m. Therefore, we have
\begin{equation} \label{omega with abs}
    \Omega = \sup_{K_1 \ll K/n^2} \sum_{m \in \mathbb{Z} } W_3 \left( \frac{m}{K_1}\right) \left| \sum_{n_3 \leqslant Y} \mathsf{a}(n_3) \mathfrak{C}_2(...) \mathcal{L}^{\pm}(...)  \right|^2,
\end{equation}
where  $K_1 \leqslant K/ n^2$, and  the character sum $\mathfrak{C}_2(...)$ was given by equation \eqref{C2}, 
\begin{equation*}
      \mathfrak{C}_2(m_1,m_2,m,n,n_3;q) = \; \sideset{}{^\star}\sum_{a \; \textrm{mod } q} S\left(\overline{a}, \pm m; \frac{q}{n}\right) \; \mathfrak{C}(m_1,m_2,a;q) e\left(\frac{-a n_3^k}{q}\right), 
\end{equation*}
and $\mathfrak{C}(...)$ is given in \eqref{char poisson}. The sum $\mathfrak{C}_2(...)$ can be simplified using Lemma \ref{gauss sum} as the sum over $\alpha_1$ and $\alpha_2$ are the Gauss sums.
\begin{align} \label{new char poisson}
   \mathfrak{C}(m_1,m_2,a,q) & = \sum_{\alpha_1, \;\alpha_2( mod \;q)} e\left(\frac{-aQ(\alpha_1, \alpha_2) + m_1\alpha_1 + m_2\alpha_2}{q}\right)  \notag\\
    &= \varepsilon^2_q \left(\left(\frac{-2 \overline{|A|a}}{q} \right) \sqrt{q}\right)^2 e\left(\frac{-\overline{4 a} Q^\star (m_1, m_2)}{q}\right) \notag \\
    &= \varepsilon_q^2\;  q \; e\left(\frac{-\overline{4 a} Q^\star (m_1, m_2) }{q}\right), 
\end{align}
for some quadratic polynomial $ Q^\star$.  If $q \equiv 2 \; \text{mod} \; 4$ or $q \equiv 0 \; \text{mod} \; 4$, we can write $\mathfrak{C}(...)$ using Lemma \ref{quadratic gauss sum} as
\begin{equation*}
    \mathfrak{C}(m_1,m_2,a;q) = 2i \; \varepsilon_{\overline{a}}^2 \;q \; e\left(\frac{\overline{a} Q^*(m_1, m_2)}{4q} \right),
\end{equation*}
which is similar to what we have in equation \eqref{new char poisson}, as the absolute value of other quantities before $q$ is $1$. Therefore, for the estimation of character sum, we have considered the former expression for $\mathfrak{C}(...)$. Consequently, we arrive at
\begin{equation} \label{C2 simplified}
    \mathfrak{C}_2(m_1,m_2,n,m,n_3 ;q) = q  \; \varepsilon_q^2  \sideset{}{^\star}\sum_{a \; \textrm{mod } q} S\left(\overline{a},  \pm m; \frac{q}{n}\right)  e\left(\frac{-\overline{4a}Q^\star (m_1, m_2)}{q} \right) e\left(\frac{-a n_3^k}{q}\right).
\end{equation}    


\subsection{Applying Poisson summation.} \label{4.5 second poisson}
We now open the absolute value square in \eqref{omega with abs}.
\begin{equation*}
    \Omega = \sup_{K_1 \ll K/n^2} \sum_{n_3 \leqslant Y} \sum_{n_3^\prime \leqslant Y} \sum_{m \in \mathbb{Z}} \mathsf{a}(n_3)\mathsf{a}(n_3^\prime)W_3\left(\frac{m}{K_1}\right) \mathfrak{C}_2(...) \overline{\mathfrak{C}_2}(...) \mathcal{L}^\pm(...) \overline{\mathcal{L}^\pm}(...),
\end{equation*}
where
\begin{multline} 
     \mathfrak{C}_2(...)\overline{\mathfrak{C}_2}(...) = q^2 \; \sideset{}{^\star}\sum_{a_1 \; \textrm{mod } q} S\left( \overline{a_1}, \pm m; \frac{q}{n}\right) e\left( \frac{-\overline{4a_1}Q^\star (m_1, m_2)}{q}\right)e\left( \frac{-a_1 n_3^k}{q}\right) \\
      \qquad \qquad \qquad \times  \;  \sideset{}{^\star}\sum_{a_2 \; \textrm{mod } q}  S\left( \overline{-a_2}, \mp m; \frac{q}{n}\right)e\left( \frac{\overline{4a_2}Q^\star (m_1, m_2)}{q}\right)e\left(\frac{a_2 n_3^{\prime k}}{q}\right).
\end{multline}
Using the change of variable $m \mapsto mq/n + j$ with $0 \leqslant j < q/n$, then applying Poisson to the $m$-sum, we arrive at
\begin{multline*}
    \Omega = \sup_{K_1 \ll K/n^2} \sum_{n_3 \leqslant Y} \sum_{n_3^\prime \leqslant Y} \mathsf{a}(n_3)\mathsf{a}(n_3^\prime)\sum_{m \in \mathbb{Z}} \sum_{j \;mod \;q/n} \mathfrak{C}_2(...)\overline{\mathfrak{C}_2}(...) \\ \times \int_{\mathbb{R}} W_3\left(\frac{wq/n + j}{K_1}\right)  \mathcal{L}^\pm(...) \overline{\mathcal{L}^\pm}(...) e(-m w) dw.
\end{multline*}
Applying the change of variable $\frac{wq/n + j}{K_1} \mapsto w$, we get
\begin{multline*}
    \Omega = \sup_{K_1 \ll K/n^2} \frac{K_1}{q/n} \sum_{n_3 \leqslant Y} \sum_{n_3^\prime \leqslant Y} \mathsf{a}(n_3)\mathsf{a}(n_3^\prime) \sum_{m \in \mathbb{Z}} \sum_{j \;mod \;q/n} e\left(\frac{m j}{q/n}\right) \mathfrak{C}_2(.,\pm j,.)\overline{\mathfrak{C}_2}(.,\mp j,.)\\ \times \int_{\mathbb{R}} W_3(w) \mathcal{L}^\pm(..,wK_1,.) \overline{\mathcal{L}^\pm}(..,wK_1,.) e\left(-\frac{ mn wK_1}{ q}\right) dw.
\end{multline*}
Define
\begin{equation} \label{character sum after cauchy}
    \mathfrak{S}: = \frac{1}{q/n} \sum_{j \;mod \;q/n} e\left(\frac{m j}{q/n}\right) \mathfrak{C}_2(.,\pm j,.)\overline{\mathfrak{C}_2}(.,\mp j,.),
\end{equation}
and 
\begin{equation} \label{zintegral}
    \mathfrak{Z} := \int_{\mathbb{R}} W_3(w) \mathcal{L}^\pm(..,wK_1,.) \overline{\mathcal{L}^\pm}(..,wK_1,.) e\left(-\frac{ mn wK_1}{ q}\right) dw. 
\end{equation}
The expression for $\Omega$ becomes
\begin{equation} \label{omega final}
    \Omega = \sup_{K_1 \ll K/n^2} K_1 \sum_{n_3 \leqslant Y} \sum_{n_3^\prime \leqslant Y} \mathsf{a}(n_3)\mathsf{a}(n_3^\prime)\sum_{m \in \mathbb{Z}} \mathfrak{S} \mathfrak{Z}.
\end{equation}
Here $\mathcal{L}^\pm(...)= \mathcal{L}^\pm(m_1,m_2,n,wK_1,n_3,q)$ which is given by
\begin{multline*}
   \mathcal{L}^\pm(...) = \frac{-2\pi^{5/2}}{\sqrt{3}} \int_{\mathbb{R}} \psi(q,u) U(u) \int_{0}^{\infty} V^{\pm}(z) z^{-1/3} e\Biggl( \frac{X zu}{q \mathcal{Q}} \pm \frac{3(X z n^2wK_1)^{1/3}}{q}\Biggr) dz \\  
     \times \iint_{\mathbb{R}^2} W_1(v_1) W_2(v_2) e\Biggl( \frac{- m_1 v_1 X^{1/2} -m_2 v_2 X^{1/2} }{q} \Biggr) e\Biggl(\frac{-uQ(v_1 X^{1/2}, v_2X^{1/2})}{q \mathcal{Q}}\Biggr) dv_1 dv_2 \\ \times e\left( \frac{-un_3^k}{q}\right)du.
\end{multline*}
Differentiating the above integral $j$ times with respect to $w$, we get
\begin{equation*}
    \frac{\partial^j}{\partial w^j} \mathcal{L}^\pm(...) \ll \left(\frac{(X n^2K_1)^{1/3}}{q} \right)^j \ll \left(\frac{(X K)^{1/3}}{q} \right)^j \ll \mathcal{Q}^{j\varepsilon},
\end{equation*}
applying integration by parts repeatedly, the integral $\mathfrak{Z}$ is negligibly small if 
\begin{equation} \label{M}
    m \gg  \frac{ q}{n K_1} := M.
\end{equation}
Using Lemma \ref{Lintegral}, we established the following bound.
\begin{lemma} \label{Zintegral}
Let $\mathfrak{Z}$ be as given in equation \eqref{zintegral}. Then 
\begin{equation*}
    \mathfrak{Z} \ll \frac{q^2}{\mathcal{Q}^2}.
\end{equation*}
\end{lemma}

\subsection{Character Sum} \label{character sum}
In this subsection, we will estimate the character sum $\mathfrak{S}$ as given in equation \eqref{character sum after cauchy}. We will consider two cases separately: when $m=0$ and when $m\neq 0$. \\

\smallskip
\textit{Case I (Zero Frequency):} We shall first estimate the character sum in the case of zero frequency $ m=0$. Let $  \mathfrak{S}_0$ denote the contribution to the character sum in this case.  Taking $ m=0$ in equation \eqref{character sum after cauchy}, we obtain
\begin{multline*}
    \mathfrak{S}_0 = \frac{1}{q/n}  q^2 \sum_{j \;\textrm{mod }q/n}  \ \sideset{}{^\star}\sum_{a_1 \; \textrm{mod } q} S\left( \overline{a_1}, \pm j; \frac{q}{n}\right) e\left( \frac{-\overline{4a_1}Q^\star (m_1, m_2) }{q}\right)e\left( \frac{-a_1 n_3^k}{q}\right) \\
      \qquad \qquad \qquad \times  \;  \sideset{}{^\star}\sum_{a_2 \; \textrm{mod } q}  S\left( -\overline{a_2}, \mp j; \frac{q}{n}\right)e\left( \frac{\overline{4a_2}Q^\star (m_1, m_2)}{q}\right)e\left(\frac{a_2 n_3^{\prime k}}{q}\right)
\end{multline*}
\begin{multline*}
   \hspace{0.2cm} = n q \sum_{j \; \textrm{mod } q/n}\; \sideset{}{^\star}\sum_{a_1 \; \textrm{mod } q}\; \; \sideset{}{^\star}\sum_{\alpha_1 \; \textrm{mod } q/n} e\left( \frac{\overline{a_1}\alpha_1 \pm j \overline{\alpha_1}}{q/n}\right) e\left( \frac{-\overline{4a_1}Q^\star (m_1, m_2) }{q}\right)e\left( \frac{-a_1 n_3^k}{q}\right) \\ \times
    \sideset{}{^\star}\sum_{a_2 \; \textrm{mod } q} \; \;\sideset{}{^\star}\sum_{\alpha_2 \; \textrm{mod } q/n} e\left( \frac{-\overline{a_2}\alpha_2 \mp j \overline{\alpha_2}}{q/n}\right) e\left( \frac{\overline{4a_2}Q^\star (m_1, m_2) }{q}\right)e\left( \frac{a_2 n_3^{\prime k}}{q}\right) 
\end{multline*}
\begin{multline*}
        = n q \; \sideset{}{^\star}\sum_{\alpha_1 \; \textrm{mod } q/n} \; \sideset{}{^\star}\sum_{\alpha_2 \; \textrm{mod } q/n} \sum_{j \; \textrm{mod } q/n} e\left( \frac{j (\pm \overline{\alpha_1} \mp \overline{\alpha_2}) }{q/n}\right) \\ \sideset{}{^\star}\sum_{a_1 \; \textrm{mod } q} \; \sideset{}{^\star}\sum_{a_2 \; \textrm{mod } q} e\left( \frac{\overline{a_1}\alpha_1 - \overline{a_2} \alpha_2}{q/n} \right) e\left( \frac{\overline{4}Q^\star (m_1, m_2) (\overline{a_2} - \overline{a_1})}{q}\right) e\left( \frac{n_3^{\prime k}a_2 - n_3^k a_1}{q}\right)
\end{multline*}
\begin{multline*}
    = q^2 \; \mathop{\sideset{}{^\star}\sum_{\alpha_1 \; \textrm{mod } q/n} \; \sideset{}{^\star}\sum_{\alpha_2 \; \textrm{mod } q/n}}_{ \pm \overline{\alpha_1} \mp \overline{\alpha_2} \equiv 0 \; \textrm{mod } q/n} \; \sideset{}{^\star}\sum_{a_1 \; \textrm{mod } q} \; \sideset{}{^\star}\sum_{a_2 \; \textrm{mod } q} e\left( \frac{\overline{a_1}\alpha_1 - \overline{a_2} \alpha_2}{q/n} \right) \\ e\left( \frac{\overline{4}Q^\star (m_1, m_2)(\overline{a_2} - \overline{a_1})}{q}\right) e\left( \frac{n_3^{\prime k}a_2 - n_3^k a_1}{q}\right).
\end{multline*}
From the above congruence condition, we obtain that $\alpha_1 = \alpha_2$. For simplicity, we can take $n=1$.
\begin{multline*}
    \mathfrak{S}_0 = q^2 \; \sideset{}{^\star}\sum_{a_1 \; \textrm{mod } q} \; \sideset{}{^\star}\sum_{a_2 \; \textrm{mod } q} \; \sideset{}{^\star}\sum_{\alpha_1 \; \textrm{mod } q} e\left( \frac{\alpha_1(\overline{a_1} - \overline{a_2})}{q}\right) e\left( \frac{\overline{4}Q^\star (m_1, m_2)(\overline{a_2} - \overline{a_1})}{q}\right) \\ \times e\left( \frac{n_3^{\prime k}a_2 - n_3^k a_1}{q}\right)
\end{multline*}
\begin{multline*}
    \hspace{0.4cm}= q^2 \; \sideset{}{^\star}\sum_{a_1 \; \textrm{mod } q} \; \sideset{}{^\star}\sum_{a_2 \; \textrm{mod } q} e\left( \frac{\overline{4}Q^\star (m_1, m_2)(\overline{a_2} - \overline{a_1})}{q}\right) e\left( \frac{n_3^{\prime k}a_2 - n_3^k a_1}{q}\right) \\ \times \sum_{d \mid (q,\overline{a_1}-\overline{a_2})} d\mu\left(\frac{q}{d}\right).
\end{multline*}
\noindent{\textit{Sub-case}(i) $n_3 = n_3^\prime$.} 
\begin{multline*}
   \mathfrak{S}_0 = q^2 \sideset{}{^\star}\sum_{a_1 \; \textrm{mod } q} \; \sideset{}{^\star}\sum_{a_2 \; \textrm{mod } q} e\left( \frac{\overline{4}Q^\star (m_1, m_2)(\overline{a_2} - \overline{a_1})}{q}\right) e\left( \frac{n_3^{k}(a_2 - a_1)}{q}\right) \\ \times \sum_{d \mid (q,\overline{a_1}-\overline{a_2})} d\mu\left(\frac{q}{d}\right)
\end{multline*}
\begin{equation*}
    \hspace{-2.7cm} \ll q^3 S( -\overline{4}Q^\star (m_1, m_2),-n_3^k;q) S(\overline{4}Q^\star (m_1, m_2),n_3^k;q) \ll q^4.
\end{equation*}

\smallskip
\noindent{\textit{Sub-case}(ii) $n_3 \neq n_3^\prime$.} We can write $q$ as $q=q_1q_2$, where $q_1$ is square full, $\mu(q_2)\neq 0$ and $(q_1,q_2)=1$. Also, write $a_1 = b_1q_2\overline{q_2} + c_1 q_1 \overline{q_1}$ and $a_2 = b_2q_2\overline{q_2} + c_2 q_1\overline{q_1}$, then we have
\begin{equation*}
    \mathfrak{S}_0(q) = \mathfrak{S}_0(q_1)\cdot \mathfrak{S}_0(q_2), 
\end{equation*}
where
\begin{multline*}
    \mathfrak{S}_0(q_1) = q_1^2 \sideset{}{^\star}\sum_{b_1 \; \textrm{mod } q_1} \ \sideset{}{^\star}\sum_{b_2 \; \textrm{mod } q_1} e\left( \frac{\overline{4q_2}Q^\star (m_1, m_2)(\overline{b_2} - \overline{b_1})}{q_1}\right) e\left( \frac{\overline{q_2}(n_3^{\prime k}b_2 - n_3^k b_1)}{q_1}\right) \\ \times \sum_{\substack{d_1 \mid q_1 \\ \overline{b_1} \equiv \overline{b_2}\  mod \ d_1}} d_1 \mu\left(\frac{q_1}{d_1}\right)
\end{multline*}
\begin{equation*}    
    \hspace{-0.6cm}\ll q_1^3 S( -\overline{4q_2}Q^\star (m_1, m_2),-n_3^k;q) S(\overline{4q_2}Q^\star (m_1, m_2),n_3^k;q) \ll q_1^4, 
\end{equation*} 
and 
\begin{multline*}
    \mathfrak{S}_0(q_2) = q_2^2 \sideset{}{^\star}\sum_{c_1 \; \textrm{mod } q_2} \ \sideset{}{^\star}\sum_{c_2 \; \textrm{mod } q_2} e\left( \frac{\overline{4q_1}Q^\star (m_1, m_2)(\overline{c_2} - \overline{c_1})}{q_2}\right) e\left( \frac{\overline{q_1}(n_3^{\prime k}c_2 - n_3^k c_1)}{q_2}\right) \\ \times \sum_{\substack{d_2 \mid q_2 \\ \overline{c_1} \equiv \overline{c_2}\  mod \ d_2}} d_2 \mu\left(\frac{q_2}{d_2}\right).
\end{multline*}
Without loss of generality, we can assume $q_2$ to be prime then $\mathfrak{S}_0(q_2)$ is bounded by
\begin{equation*}
    \ll q_2^3 \sideset{}{^\star}\sum_{c_1 \; \textrm{mod } q_2} e\left( \frac{\overline{q_1}c_1(n_3^{\prime k} - n_3^k)}{q_2}\right) \ll q_2^3(n_3^{\prime k} - n_3^k,q_2).
\end{equation*}
We record the preceding analysis in the following lemma. 
\begin{lemma} \label{char sum zero freq}
    The character sum given as in equation \eqref{character sum after cauchy} for the case $m=0$ is bounded by
    \begin{equation*}
        \mathfrak{S}_{0} \ll \begin{cases}
            q^4/n, & \text{if } n_3 = n_3^\prime\\
            q^3q_1(n_3^{\prime k} - n_3^k, q_2)/n , & \ \text{otherwise} 
        \end{cases}, 
    \end{equation*}
where $q=q_1q_2$, $q_1$ is square full, $\mu(q_2)\neq 0$ and $(q_1,q_2)=1$.    
\end{lemma}
    
\smallskip
\textit{ Case II (Non-zero frequencies):}  In this case, we denote $\mathfrak{S}_{\neq 0}$ as the contribution of character sum given in equation \eqref{character sum after cauchy} when $m\neq 0$. We have 
\begin{multline*}
    \mathfrak{S}_{\neq 0}(q) = q n \hspace{-0.3cm}\sum_{j \;mod \;q/n}  e\left( \frac{m j}{q/n}\right) \;  \sideset{}{^\star}\sum_{a_1 \; \textrm{mod } q} S\left( \overline{a_1}, \pm j; \frac{q}{n}\right) e\left( \frac{-\overline{4a_1}Q^\star (m_1, m_2)}{q}\right)e\left( \frac{-a_1 n_3^k}{q}\right) \\
      \qquad \qquad \qquad \times  \;  \sideset{}{^\star}\sum_{a_2 \; \textrm{mod } q}  S\left( -\overline{a_2}, \mp j; \frac{q}{n}\right)e\left( \frac{\overline{4a_2}Q^\star (m_1, m_2) }{q}\right)e\left(\frac{a_2 n_3^{\prime k}}{q}\right).
\end{multline*}
We write $q$ as $q= q_1 q_2 q_3$ where $n \mid q_1 q_2 \mid n^{\infty} $ and $(q_3,n)=1$. The above sum can be written as follows
\begin{equation} \label{sum split in q1q2q3}
    \mathfrak{S}_{\neq 0}(q) = \mathfrak{S}_{\neq 0}(q_1 q_2) \cdot \mathfrak{S}_{\neq 0}(q_3),
\end{equation}
where
\begin{multline*}
   \mathfrak{S}_{\neq 0}(q_1q_2) = q_1 q_2 n \hspace{-0.4cm}\sum_{j_1 \;mod \;q_1 q_2/n} e\left( \frac{m j_1 \overline{q_3}}{q_1 q_2/n}\right) \\ \times \sideset{}{^\star}\sum_{a_1^{\prime} \; \textrm{mod } q_1 q_2} \; \sideset{}{^\star}\sum_{\alpha_1 \; \textrm{mod } q_1 q_2/n} e\left( \frac{\overline{q_3}(\overline{a_1^\prime}\alpha_1 \pm j_1 \overline{\alpha_1})}{q_1 q_2/n} \right) e\left( \frac{-\overline{4a_1^\prime} \overline{q_3}Q^\star (m_1, m_2) }{q_1 q_2}\right)e\left( \frac{-a_1^\prime \overline{q_3} n_3^k}{q_1 q_2}\right) \\ \times \sideset{}{^\star}\sum_{a_2^\prime \; \textrm{mod } q_1q_2} \; \ \sideset{}{^\star}\sum_{\alpha_2 \; \textrm{mod } q_1 q_2/n} e\left( \frac{\overline{q_3}(-\overline{a_2^\prime}\alpha_2 \mp j_1 \overline{\alpha_2})}{q_1 q_2/n}\right) e\left( \frac{\overline{4a_2^\prime} \overline{q_3}Q^\star (m_1, m_2)}{q_1q_2}\right)e\left(\frac{a_2^\prime \overline{q_3} n_3^{\prime k}}{q_1q_2}\right)
\end{multline*}
\begin{multline*}
   = q_1 q_2 n \sideset{}{^\star}\sum_{a_1^\prime \; \textrm{mod } q_1q_2} \; \sideset{}{^\star}\sum_{a_2^\prime \; \textrm{mod } q_1q_2} \; e\left( \frac{\overline{q_3}(a_2^\prime  n_3^{\prime k} - a_1^\prime n_3^k)}{q_1q_2}\right) e\left( \frac{\overline{4} \overline{q_3}Q^\star (m_1, m_2)( \overline{a_2^\prime} - \overline{a_1^\prime})}{q_1q_2}\right) \\ \times \sideset{}{^\star}\sum_{\alpha_1 \; \textrm{mod } q_1q_2/n} \; \sideset{}{^\star}\sum_{\alpha_2 \; \textrm{mod } q_1q_2/n} e\left( \frac{\overline{q_3}(\overline{a_1^\prime} \alpha_1 - \overline{a_2^\prime} \alpha_2)}{q_1q_2/n}\right) \\ \times \sum_{j_1 \;mod \;q_1q_2/n} e\left( \frac{\overline{q_3}j_1( m \pm \overline{\alpha_1} \mp \overline{\alpha_2})}{q_1q_2/n}\right)
\end{multline*}
\begin{multline*}
    = (q_1q_2)^2 \sideset{}{^\star}\sum_{a_1^\prime \; \textrm{mod } q_1q_2} \; \sideset{}{^\star}\sum_{a_2^\prime \; \textrm{mod } q_1q_2} \; e\left( \frac{\overline{q_3}(a_2^\prime n_3^{\prime k} - a_1^\prime n_3^k)}{q_1q_2}\right) e\left( \frac{\overline{4} \overline{q_3}Q^\star (m_1, m_2)( \overline{a_2^\prime} - \overline{a_1^\prime})}{q_1q_2}\right) \\ \times \mathop{\sideset{}{^\star}\sum_{\alpha_1 \; \textrm{mod } q_1q_2/n}\; \; \sideset{}{^\star}\sum_{\alpha_2 \; \textrm{mod } q_1q_2/n}}_{\pm \overline{\alpha_1} \mp \overline{\alpha_2} \equiv -m \; \textrm{mod } q_1q_2/n}  e\left( \frac{\overline{q_3}(\overline{a_1^\prime} \alpha_1 - \overline{a_2^\prime} \alpha_2)}{q_1q_2/n}\right)
\end{multline*}
\begin{multline*}
    = (q_1q_2)^2 \sideset{}{^\star}\sum_{a_1^\prime \; \textrm{mod } q_1q_2} \; \sideset{}{^\star}\sum_{a_2^\prime \; \textrm{mod } q_1q_2} \; e\left( \frac{\overline{q_3}(a_2^\prime n_3^{\prime k} - a_1^\prime n_3^k)}{q_1q_2}\right) e\left( \frac{\overline{4} \overline{q_3}Q^\star (m_1, m_2)( \overline{a_2^\prime} - \overline{a_1^\prime})}{q_1q_2}\right) \\ \times \sideset{}{^\star}\sum_{\alpha_1 \; \textrm{mod } q_1q_2/n} e\left( \frac{\overline{q_3}(\overline{a_1^\prime} \alpha_1 - \overline{a_2^\prime}(\overline{\overline{\alpha_1} +m}))}{q_1q_2/n}\right)
\end{multline*}
\begin{multline*}
    = (q_1 q_2)^2 \sideset{}{^\star}\sum_{\alpha_1 \; \textrm{mod } q_1q_2/n} S\left(-\overline{q_3} n_3^k, \overline{q_3}(\alpha_1 - \overline{4} Q^\star (m_1, m_2)); q_1q_2\right) \\ \times S\left(\overline{q_3} n_3^{\prime k}, -\overline{q_3}((\overline{\overline{\alpha_1} +m}) - \overline{4}Q^\star (m_1, m_2)); q_1q_2\right).
\end{multline*}
Using Weil's bound for Kloosterman sums and executing $ \alpha_1$ sum trivially, we obtain that
\begin{equation} \label{sum with mod q1q2}
    \mathfrak{S}_{\neq 0}(q_1q_2) \ll \frac{(q_1 q_2)^4}{n}.
\end{equation}
Next, we will bound the other sum in equation \eqref{sum split in q1q2q3}, which is given as 
\begin{align} \label{sum with mod q3}
    \mathfrak{S}_{\neq 0}(q_3) &= q_3 \sum_{j_2 \; \textrm{mod } q_3} e\left( \frac{m j_2\overline{(q_1 q_2/n)}}{q_3}\right) \notag \\ 
    &\times \sideset{}{^\star}\sum_{a_1\dblprime \; \textrm{mod } q_3} S\left(\overline{a_1\dblprime} \overline{\frac{q_1q_2}{n}}, \pm j_2 \overline{\frac{q_1q_2}{n}}; q_3 \right) e\left(  \frac{-\overline{4} \overline{a_1\dblprime} \overline{q_1q_2}Q^\star (m_1, m_2)}{q_3}\right)  e\left( \frac{-a_1\dblprime \overline{q_1q_2} n_3^k}{q_3}\right) \notag \\  &\times \sideset{}{^\star}\sum_{a_2\dblprime \; \textrm{mod } q_3} S\left( -\overline{a_2\dblprime} \overline{\frac{q_1q_2}{n}}, \mp j_2 \overline{\frac{q_1q_2}{n}}; q_3 \right) e\left( \frac{\overline{4} \overline{a_2\dblprime} \overline{q_1q_2} Q^\star (m_1, m_2)}{q_3}\right) e\left( \frac{a_2\dblprime \overline{q_1q_2} n_3^{\prime k}}{q_3}\right).
\end{align} 
To proceed further, we split $q_3$ as $q_3 = q_3^\prime q_3\dblprime$, where $\mu(q_3^\prime) \neq 0$, $q_3\dblprime$ is square full, and $(q_3^\prime, q_3\dblprime) = 1$. Now, the sum $\mathfrak{S}_{\neq 0}(q_3)$ can be expressed as
\begin{equation} \label{sum split in q3}
    \mathfrak{S}_{\neq 0}(q_3) = \mathfrak{S}_{\neq 0}(q_3^\prime) \cdot \mathfrak{S}_{\neq 0}(q_3\dblprime).
\end{equation}
Using Weil's bound and a similar approach as before, the sum with moduli $q_3\dblprime$ can be straightforwardly bounded. Specifically, we attain the following
\begin{equation}
    \mathfrak{S}_{\neq 0}(q_3\dblprime) \ll {q_3\dblprime}^4.
\end{equation}
At this point, the only remaining sum to be addressed is $\mathfrak{S}_{\neq 0}(q_3^\prime)$. Since $q_3^\prime$ is a product of primes, it becomes feasible to treat it as a prime itself. 
\begin{multline*}
    \mathfrak{S}_{\neq 0}(q_3^\prime) = q_3^\prime \sideset{}{^\star}\sum_{a_1\dblprime \; \textrm{mod } q_3^\prime} \ \sideset{}{^\star}\sum_{\beta_1 \; \textrm{mod } q_3^\prime}  e\left(  \frac{-\overline{4 a_1\dblprime q_1q_2}Q^\star (m_1, m_2)}{q_3^\prime}\right) e\left( \frac{-a_1\dblprime \overline{q_1q_2} n_3^k}{q_3^\prime}\right)  \\ \hspace{9cm}\times  e\left( \frac{\overline{a_1\dblprime} \overline{(q_1q_2/n)} \beta_1}{q_3^\prime}\right) \\ \times \sideset{}{^\star}\sum_{a_2\dblprime \; mod \;q_3^\prime}  \; \; \sideset{}{^\star}\sum_{\beta_2 \; mod \;q_3^\prime} e\left( \frac{-\overline{a_2\dblprime} \overline{(q_1q_2/n)} \beta_2}{q_3^\prime}\right) e\left( \frac{\overline{4} \overline{a_2\dblprime} \overline{q_1q_2} Q^\star (m_1, m_2) }{q_3^\prime}\right) e\left( \frac{a_2\dblprime \overline{q_1q_2} n_3^{\prime k}}{q_3^\prime}\right) \\ \hspace{0.9cm}\times \sum_{j_2 \; \textrm{mod } q_3^\prime} e\left( \frac{j_2\overline{(q_1 q_2/n)} ( m \pm  \overline{\beta_1} \mp \overline{\beta_2} )}{q_3^\prime}\right).
\end{multline*}
Executing the sum over $j_2$, then determining $\beta_2$ in terms of $\beta_1$ using the congruence $\pm \overline{\beta_1} \mp \overline{\beta_2} \equiv -m \; \textrm{mod } q_3^\prime$, we arrive at
\begin{multline*}
    = q_3^{\prime 2} \sideset{}{^\star}\sum_{a_1\dblprime \; \textrm{mod } q_3^\prime} \; \sideset{}{^\star}\sum_{a_2\dblprime \; mod \;q_3^\prime} e\left(  \frac{-\overline{4} \overline{a_1\dblprime} \overline{q_1q_2}Q^\star (m_1, m_2) }{q_3^\prime}\right)  e\left( \frac{-a_1\dblprime \overline{q_1q_2} n_3^k}{q_3^\prime}\right)  e\left( \frac{a_2\dblprime \overline{q_1q_2} n_3^{\prime k}}{q_3^\prime}\right) \\ \times e\left( \frac{\overline{4} \overline{a_2\dblprime} \overline{q_1q_2} Q^\star (m_1, m_2) }{q_3^\prime}\right) \sideset{}{^\star}\sum_{\beta_1 \; \textrm{mod } q_3^\prime} e\left( \frac{\overline{(q_1q_2/n)}( \overline{a_1\dblprime} \beta_1 - \overline{a_2\dblprime} (\overline{\overline{\beta_1} \pm m })}{q_3^\prime} \right),
\end{multline*}
Substituting  the following change of variable we obtain
\begin{equation*}
    \beta = \pm m \beta_1 + 1 \Rightarrow \beta -1 = \pm m \beta_1 \Rightarrow \beta_1 = \pm \overline{m}(\beta - 1),
\end{equation*}
we get
\begin{equation*}
    \overline{\beta_1} \pm m = \overline{\beta_1}( 1 \pm m \beta_1) = \pm m (\overline{\beta - 1})\beta,
\end{equation*}
and hence
\begin{equation*}
    -(\overline{\overline{\beta_1} \pm m}) = \mp \overline{\beta m}(\beta - 1) = \pm \overline{m}(\overline{\beta} - 1).
\end{equation*}
Thus, we reached at
\begin{multline} \label{q3prime after variable change}
    \mathfrak{S}_{\neq 0}(q_3^\prime) = q_3^{\prime 2} \sideset{}{^\star}\sum_{a_1\dblprime \; \textrm{mod } q_3^\prime} \; e\left(  \frac{-\overline{4} \overline{a_1\dblprime} \overline{q_1q_2}Q^\star (m_1, m_2)}{q_3^\prime}\right)  e\left( \frac{-a_1\dblprime \overline{q_1q_2} n_3^k}{q_3^\prime}\right)  \\ \times \sideset{}{^\star}\sum_{a_2\dblprime \; mod \;q_3^\prime} e\left( \frac{\overline{4} \overline{a_2\dblprime} \overline{q_1q_2} Q^\star (m_1, m_2) }{q_3^\prime}\right) e\left( \frac{a_2\dblprime \overline{q_1q_2} n_3^{\prime k}}{q_3^\prime}\right) \\ \times \sideset{}{^\star}\sum_{\beta \; \textrm{mod } q_3^\prime} e\left( \frac{ \pm \overline{m} \overline{(q_1q_2/n)}\left( \overline{a_1\dblprime} (\beta -1) + \overline{a_2\dblprime} (\overline{\beta} - 1) \right) }{q_3^\prime} \right).
\end{multline}
To estimate the above sum, let $c_1 = -\overline{q_1q_2} n_3^k$, $c_2 = -\overline{4} \overline{q_1q_2} Q^\star (m_1, m_2) \mp \overline{m} (\overline{q_1q_2/n})$, $c_3 = \overline{q_1q_2} n_3^{\prime k}$, $c_4 = \overline{4} \overline{q_1q_2} Q^\star (m_1, m_2) \mp \overline{m} ( \overline{q_1q_2/n})$ and $c_5 = \pm \overline{m}(\overline{q_1q_2/n})$. The above sum can be reformulated as
\begin{eqnarray} \label{q3prime before matrix}
   \mathfrak{S}_{\neq 0}(q_3^\prime)/ q_3^{\prime 2} &=& \sideset{}{^\star}\sum_{\beta \; \textrm{mod } q_3^\prime} S\left( c_1, c_2 + c_5\beta ;q_3^\prime \right) S\left( c_3, c_4 + c_5 \overline{\beta}; q_3^\prime \right) \notag\\
    &=& \sideset{}{^\star}\sum_{\beta \; \textrm{mod } q_3^\prime} S\left( 1, c_1(c_2 + c_5\beta) ;q_3^\prime \right) \overline{S}\left( 1, c_3(c_4 + c_5 \overline{\beta}); q_3^\prime \right),
\end{eqnarray}
if $(n_3^k n_3^{\prime k}, q_3^\prime) = 1$. Character sums of this nature have been explored in \cite{DF} and \cite{FKM} employing the $\ell$-adic techniques developed by Deligne and Katz. By incorporating their ideas, square root cancellation can be achieved in the sum \eqref{q3prime before matrix}. However, here we are showcasing ideas developed by D\k{a}browski and Fisher. Consider the linear fractional transformations
\begin{equation*}
    \gamma_1 = \begin{pmatrix} 
        c_1c_5 & c_1c_2 \\
        0 & 1
    \end{pmatrix} \qquad  \text{and} \qquad \gamma_2 = \begin{pmatrix}
        c_3 c_4 & c_3 c_5 \\
        1 & 0
    \end{pmatrix},
\end{equation*}
we can recast equation \eqref{q3prime before matrix} as
\begin{eqnarray*}
    \mathfrak{S}_{\neq 0}(q_3^\prime)/ q_3^{\prime 2} &=& \sideset{}{^\star}\sum_{\beta \; \textrm{mod } q_3^\prime} S(1, \gamma_1(\beta); q_3^\prime) \overline{S}(1, \gamma_2(\beta); q_3^\prime) \\
    &=& \sideset{}{^\star}\sum_{\beta \; \textrm{mod } q_3^\prime} S(1, \beta; q_3^\prime) \overline{S}(1, \gamma_2\gamma_1^{-1}(\beta); q_3^\prime),
\end{eqnarray*}
since $\text{det}(\gamma_i) \neq 0 \; \textrm{mod } q_3^\prime$. Propositions 3.4 in \cite{DF} implies the following. Given $\gamma = \begin{pmatrix}
    a & b \\ c & d
\end{pmatrix}$ such that $ad-bc \neq 0 \; mod \; p$, then 
\begin{equation*}
    \sideset{}{^\star}\sum_{\alpha \; mod \; p} S(1, \alpha ; p) \overline{S}(1, \gamma(\alpha) ; p) \ll p^{3/2}, 
\end{equation*}
where the implied constant depends on linear fractional transformation $\gamma$. In the present case, we have
\begin{equation*}
    \gamma_2 \gamma_1^{-1} = \frac{1}{c_1c_5}\begin{pmatrix}
        c_3c_4 & c_1c_3(c_5^2 - c_2c_4) \\
        1 & - c_1 c_2
    \end{pmatrix}.
\end{equation*}
Hence, we obtain our desired bound
\begin{equation} \label{bound of q3prime1}
    \mathfrak{S}_{\neq 0}(q_3^\prime) \ll q_3^{\prime \; 7/2}, \quad \text{if} \ (n_3^k n_3^{\prime k}m , q_3^\prime)=1.
\end{equation}
If $n_3^k \equiv 0 \; \textrm{mod } q_3^\prime$ and $n_3^{ \prime k} \not\equiv 0 \; \textrm{mod } q_3^\prime$, equation \eqref{q3prime after variable change} becomes 
\begin{align*}
    \mathfrak{S}_{\neq 0}(q_3^\prime) &= q_3^{\prime 2} \sideset{}{^\star}\sum_{\beta \; \textrm{mod } q_3^\prime} S\left( c_3, c_4 + c_5\overline{\beta}; q_3^\prime \right) \sideset{}{^\star}\sum_{a_1\dblprime \; \textrm{mod } q_3^\prime} e\left( \frac{\overline{a_1\dblprime}( c_2 +c_5\beta)}{q_3^\prime}\right) \\
    &\ll q_3^{\prime 3} \sideset{}{^\star}\sum_{\substack{\beta \; \textrm{mod } q_3^\prime \\ c_2 +c_5\beta \equiv 0 \ \textrm{mod } q_3^\prime}} S\left( c_3, c_4 + c_5\overline{\beta}; q_3^\prime \right),
\end{align*}
from the congruence, $\beta$ can be determined modulo $q_3^\prime$. Finally, using Weil's bound for Kloosterman sum, we have
\begin{equation}
    \mathfrak{S}_{\neq 0}(q_3^\prime) \ll q_3^{\prime 7/2}.
\end{equation}
An equivalent bound can be derived for the other case.  Now, if $n_3^k$ and $n_3^{\prime k} \equiv 0 \; \textrm{mod } q_3^\prime$, the sum reduces to
\begin{align*}
    \mathfrak{S}_{\neq 0}(q_3^\prime) &= q_3^{\prime 2} \sideset{}{^\star}\sum_{\beta \; \textrm{mod } q_3^\prime} \ \sideset{}{^\star}\sum_{a_1\dblprime \; \textrm{mod } q_3^\prime} e\left(\frac{\overline{a_1\dblprime}(c_2 + c_5\beta)}{q_3^\prime}\right) \sideset{}{^\star}\sum_{a_2\dblprime \; \textrm{mod } q_3^\prime} e\left(\frac{\overline{a_2\dblprime}(c_4 + c_5\overline{\beta})}{q_3^\prime}\right)\\
    &\ll q_3^{\prime 3} \sideset{}{^\star}\sum_{\substack{\beta \; \textrm{mod } q_3^\prime \\ c_4 + c_5\overline{\beta}\equiv 0 \ \textrm{mod } q_3^\prime}} \ \sideset{}{^\star}\sum_{a_1\dblprime \; \textrm{mod } q_3^\prime} e\left(\frac{\overline{a_1\dblprime}(c_2 + c_5\beta)}{q_3^\prime}\right) \\
    &\ll q_3^{\prime 3} \sum_{d \mid (q_3^\prime, c_2 - c_5^2 \overline{c_4})} d \mu\left(\frac{q_3^\prime}{d}\right) \ll q_3^{\prime 3}.
\end{align*}  
Consider another case where $q_3^\prime \mid m$, resulting in a summation resembling the zero-frequency case. Thus,
\begin{equation} \label{bound of q3prime2}
    \mathfrak{S}_{\neq 0}(q_3^\prime) \ll q_3^{\prime \; 4}, \ \ \text{if} \ q_3^\prime \mid m.
\end{equation}
This brings us to the following conclusion after combining all the above results deduced in equations \eqref{sum split in q1q2q3} - \eqref{bound of q3prime2}:
\begin{lemma} \label{char sum non zero}
The character sum in the non-zero frequency case is dominated by
    \begin{equation}
        \mathfrak{S}_{\neq 0}(q) \ll \begin{cases}
           q^{7/2} (q_1 q_2 q_3\dblprime)^{1/2}/n, & \text{if} \; \;(q_3^\prime, n_3^k n_3^{\prime k} m)=1 \;\text{or} \\ &\; \hspace{-0.1cm}\text{if} \ (q_3^\prime , m)=1 \; \text{and} \; (q_3^\prime, n_3^k n_3^{\prime k} ) \neq 1\\
           q^4/n, & \text{if} \; \; q_3^\prime/m 
        \end{cases},
    \end{equation}
    where $q = q_1 q_2 q_3^\prime q_3\dblprime$ with $n \mid q_1 q_2 \mid n^{\infty} $, $(q_3^\prime q_3\dblprime,n)=1$, $\mu(q_3^\prime) \neq 0$ and $q_3\dblprime$ is square full.
\end{lemma}

\subsection{Zero Frequency}
In this subsection, we will estimate the contribution of the zero frequency $m=0$ to $\Omega$ denoted by $\Omega_0$ in \eqref{omega final}. This will allow us to estimate its total contribution to $\mathcal{S}_{k, E}(X)$. We have the following lemma.

\begin{lemma} \label{zero freq s}
Let $\mathcal{S}_{k,E}(X)$ be as in \eqref{final Sk(X)}. The total contribution of the zero frequency $m=0$ to $\mathcal{S}_{k,E}(X)$ is dominated by:
\begin{equation*}
\mathcal{S}_{k,E}(X) \ll  \frac{X^{5/3+ \varepsilon} Y^{1/2} K^{2/3} }{\mathcal{Q}^2}.
\end{equation*}
\end{lemma}
\begin{proof}
On substituting bounds for $\mathfrak{Z}$ and $\mathfrak{S}_0$ from Lemma \ref{Zintegral} and \ref{char sum zero freq}, respectively into \eqref{omega final}, we see that
\begin{eqnarray*}
    \Omega_0 &\ll& \sup_{K_1 \ll K/n^2} K_1 \sum_{n_3 \leqslant Y} \left|\mathsf{a}(n_3) \mathsf{a}(n_3)\right|  \frac{q^4}{n} \frac{D^2}{\mathcal{Q}^2} + K_1 \mathop{\sum_{n_3 \leqslant Y} \sum_{n_3^{\prime} \leqslant Y}}_{n_3^k \equiv n_3^{\prime k} \ \textrm{mod } q_2} \left|\mathsf{a}(n_3) \mathsf{a}(n_3^\prime)\right|\frac{q^3q_1}{n}\frac{D^2}{\mathcal{Q}^2}\\
    &\ll& \frac{K Y D^2 q^4}{ n^3\mathcal{Q}^2} + \frac{KY^2D^2q^4}{n^3q_2^2\mathcal{Q}^2} \ll \frac{K Y D^2 q^4}{ n^3\mathcal{Q}^2}. 
\end{eqnarray*}
Using the above bound of $\Omega_0$ in \eqref{final Sk(X)}, we arrive at
\begin{align*}
    &\sup_{D \ll \mathcal{Q}} \frac{X^{5/3+\varepsilon}}{\mathcal{Q} D^4} \sum_{q \sim D} \sum_{\substack{n \ll D \\ n \mid q}} n^{1/3}  \Theta^{1/2} \left(  \frac{K Y D^2 q^4}{ n^3\mathcal{Q}^2}\right)^{1/2} \\
    &\ll \sup_{D \ll \mathcal{Q}} \frac{X^{ 5/3 + \varepsilon} Y^{1/2}  K^{1/2} D}{\mathcal{Q}^2 D^4} \sum_{q \sim D} q^2 \sum_{\substack{n \ll D \\ n \mid q}} \frac{1}{n^{7/6}} \Theta^{1/2} \\
    &\ll \sup_{D \ll \mathcal{Q}} \frac{X^{ 5/3 + \varepsilon} Y^{1/2}  K^{1/2}}{\mathcal{Q}^2D^3} \sum_{q \sim D} q^2 \sum_{\substack{n \ll D \\ n \mid q}} \frac{1}{n^{7/6}} \Theta^{1/2} .
\end{align*}
Consider
\begin{eqnarray*}
    \sum_{\substack{n \ll D \\ n \mid q}} \frac{\Theta^{1/2}}{n^{7/6}} &=&  \sum_{\substack{n \ll D \\ n \mid q}} \frac{1}{n^{7/6}} \left( \sum_{m \ll K/n^2} \frac{|B(n,m)|^2}{m^{2/3}} \right)^{1/2} \\
    &\ll& \Bigg(\sum_{\substack{n \ll D \\ n \mid q}} \frac{1}{n} \Bigg)^{1/2} \left(\mathop{\sum\sum}_{n^2m \ll K} \frac{|B(n,m)|^2}{(n^2m)^{2/3}}\right)^{1/2} \ll K^{1/6}.
\end{eqnarray*}
Using this bound, we get
\begin{equation*}
    \mathcal{S}_{k,E}(X) \ll \frac{X^{5/3+ \varepsilon} Y^{1/2} K^{2/3} }{\mathcal{Q}^2}. 
\end{equation*}
This proves the bound stated in the lemma.
\end{proof}

\subsection{Non-Zero Frequencies} In this subsection, we will determine a bound for $\Omega$ in \eqref{omega final} when $m \neq 0$. The contribution of $\Omega$ in this case is denoted by $\Omega_{\neq 0}$, which will ultimately enable us to estimate our main sum $\mathcal{S}_{k,E}(X)$ given in \eqref{final Sk(X)}. The lemma that follows gives us a bound for the sum taken over square full integers, which will be required to prove the bound for $\mathcal{S}_{k, E}(X)$.
\begin{lemma} \label{square full}
Let $n$ and $ X > 1$. We have
    \begin{equation*}
        \sum_{\substack{n \leqslant X \\ n \text{-square full}}} 1 \ll X^{1/2},
    \end{equation*}
where $X$ is sufficiently large.    
\end{lemma}
\begin{proof}
    Let $n>1$ be a square full integer. We can always write $n$ as $n=n_1^2 n_2$ with $n_2\mid n_1$. Therefore, we can split the sum over $n$ as
    \begin{equation*}
        \sum_{n_2\leqslant X} \sum_{n_1\leqslant \sqrt{\frac{X}{n_2}} } 1
        \leqslant \sum_{n_2\leqslant X} \sum_{n_3 \leqslant \frac{1}{n_2}\sqrt{\frac{X}{n_2}}} 1 \leqslant \sqrt{X} \sum_{n_2\leqslant X} \frac{1}{n_2^{3/2}} \ll \sqrt{X}.
    \end{equation*}
\end{proof}
\noindent Subsequently, we establish the following lemma.
\begin{lemma} \label{nonzero freq s}
    Let $\mathcal{S}_{k,E}(X)$ be as in \eqref{final Sk(X)}. The total contribution of the non-zero frequencies $m\neq0$ to $\mathcal{S}_{k,E}(X)$ is dominated by:
    \begin{equation*}
    \mathcal{S}_{k,E}(X) \ll \frac{X^{5/3 + \varepsilon} Y K^{1/6}}{\mathcal{Q}^{7/4}}.
\end{equation*}
\end{lemma}
\begin{proof}
    On substituting the bounds for the integral and the character sum from Lemma \ref{Zintegral} and Lemma \ref{char sum non zero}, respectively into \eqref{omega final}. We will first bound $\Omega$, where we denote the contribution of $\Omega$ to $\mathcal{S}_{k,E}(X)$ for the case when $(q_3^\prime, n_3^k n_3^{\prime k} m)=1$ or if $(q_3^\prime, m)=1$ and $ (n_3^k n_3^{\prime k}, q_3^\prime) \neq 1$ as $\Omega^1_{\neq 0}$, for the subsequent case as $\Omega_{\neq 0}^2$. Let us begin with the first case, i.e.,
\begin{align*}
    \Omega^1_{\neq0} &\ll \sup_{K_1\ll K/n^2} K_1  \sum_{n_3 \leqslant Y} \sum_{n_3^\prime \leqslant Y} \left|\mathsf{a}(n_3)\right|\left|\mathsf{a}(n_3^\prime)\right| \sum_{m \ll M} \frac{q^{7/2} (q_1 q_2 q_3\dblprime)^{1/2}}{n} \frac{D^2}{\mathcal{Q}^2} \\
    &\ll \frac{D^2 Y^2 q^{9/2} (q_1 q_2 q_3\dblprime)^{1/2}}{n^2 \mathcal{Q}^2},
\end{align*}
where $M$ is defined in equation \eqref{M}. Analogously, we have
\begin{align*}
    \Omega^2_{\neq 0} &\ll \sup_{K_1\ll K/n^2} K_1  \sum_{n_3 \leqslant Y} \sum_{n_3^\prime \leqslant Y} \left|\mathsf{a}(n_3)\right|\left|\mathsf{a}(n_3^\prime)\right| \sum_{m^* \ll M/q_3^\prime} \frac{q^4}{n} \frac{D^2}{\mathcal{Q}^2} \\
    &\ll \frac{D^2 Y^2 q^4 q_1q_2q_3\dblprime }{n^2 \mathcal{Q}^2}.
\end{align*} 
On substituting these bounds in \eqref{final Sk(X)}, we get
\begin{align*}
    \mathcal{S}^1_{k,E}(X) &\ll \sup_{D \ll \mathcal{Q}} \frac{X^{5/3 + \varepsilon} }{\mathcal{Q}D^4} \sum_{\substack{n \ll D \\ n \mid q} } n^{1/3} \sum_{q \sim D} \Theta^{1/2} \left(\frac{D^2 Y^2 q^{9/2} (q_1 q_2 q_3\dblprime)^{1/2}}{n^2 \mathcal{Q}^2  }\right)^{1/2} \\
    &\ll \sup_{D \ll \mathcal{Q}} \frac{X^{5/3 + \varepsilon} Y }{\mathcal{Q}^2 D^3} \sum_{\substack{n \ll D \\ n \mid q} } \frac{1}{n^{2/3}} \sum_{q \sim D} \Theta^{1/2} q^{9/4} (q_1 q_2 q_3\dblprime)^{1/4} \\
    &\ll \sup_{D \ll \mathcal{Q}} \frac{X^{5/3 + \varepsilon} Y }{\mathcal{Q}^2 D^3}  \sum_{\substack{n \ll D \\ n \mid q_1q_2} } \frac{(q_1q_2)^{5/2}}{n^{2/3}} \Theta^{1/2} \sum_{q_3^\prime \sim \frac{D}{q_1q_2}} q_3^{\prime \; 9/4} \sum_{q_3\dblprime \sim \frac{D}{q_1q_2q_3^\prime} } {q_3\dblprime}^{5/2}.
\end{align*}
Using Lemma \ref{square full} in the last sum, we get
\begin{align*}
  \mathcal{S}^1_{k,E}(X) &\ll \sup_{D \ll \mathcal{Q}} \frac{X^{5/3 + \varepsilon} Y D^{1/2}}{\mathcal{Q}^2} \sum_{\substack{n \ll D \\ n \mid q_1q_2} } \frac{1}{n^{2/3}q_1q_2} \Theta^{1/2} \sum_{q_3^\prime \sim \frac{D}{q_1q_2}} \frac{1}{q_3^{\prime 5/4}} \\
    &\ll \sup_{D \ll \mathcal{Q}} \frac{X^{5/3 + \varepsilon} Y D^{1/4}}{\mathcal{Q}^2} \sum_{\substack{n \ll D \\ n \mid q_1q_2} } \frac{1}{n^{2/3}(q_1q_2)^{3/4}} \Theta^{1/2} \\
    &\ll \frac{X^{5/3 + \varepsilon} Y}{\mathcal{Q}^{7/4}} \sum_{n \ll D} \frac{\Theta^{1/2}}{n^{17/12}} \\
    &\ll \frac{X^{5/3 + \varepsilon} Y K^{1/6}}{\mathcal{Q}^{7/4}}.
\end{align*}
Note that
\begin{align*}
    \sum_{n \ll D} \frac{\Theta^{1/2}}{n^{17/12}} &= \sum_{n \ll D} \frac{1}{n^{17/12}} \left( \; \sum_{n^2 m \ll K} \frac{|B(n,m)|^2}{m^{2/3}}\right)^{1/2} \\
    &\ll \left(\sum_{n\ll D} \frac{1}{n^{3/2}}\right)^{1/2} \left( \; \mathop{\sum\sum}_{n^2 m \ll K} \frac{|B(n,m)|^2}{(n^2 m)^{2/3}}\right)^{1/2} \ll K^{1/6}.
\end{align*}
Similarly, we can bound the sum in the second case as follows:
\begin{align*}
    \mathcal{S}^2_{k,E}(X) &\ll \sup_{D \ll \mathcal{Q}} \frac{X^{5/3 + \varepsilon} }{\mathcal{Q}D^4} \sum_{\substack{n \ll D \\ n \mid q} } n^{1/3} \sum_{q \sim D} \Theta^{1/2} \left(\frac{D^2 Y^2 q^4 q_1q_2q_3\dblprime }{n^2 \mathcal{Q}^2} \right)^{1/2} \\
    &\ll \sup_{D \ll \mathcal{Q}} \frac{X^{5/3 + \varepsilon} Y D^{7/2}}{\mathcal{Q}^2 D^3} \sum_{\substack{n \ll D \\ n \mid q_1q_2} } \frac{1}{n^{2/3} q_1q_2} \Theta^{1/2} \sum_{q_3^\prime \sim \frac{D}{q_1q_2}} \frac{1}{q_3^{\prime 3/2}}  \\
    &\ll \frac{X^{5/3 + \varepsilon} Y }{\mathcal{Q}^2} \sum_{\substack{n \ll D \\ n \mid q_1q_2} } \frac{\Theta^{1/2}}{n^{2/3} (q_1q_2)^{1/2}}  \\
    &\ll \frac{X^{5/3 + \varepsilon} Y}{\mathcal{Q}^{2}} \sum_{n \ll D} \frac{\Theta^{1/2}}{n^{7/6}}\\
    &\ll \frac{X^{5/3 + \varepsilon} Y K^{1/6}}{\mathcal{Q}^{2}}.
\end{align*}
Also, note that
\begin{eqnarray*}
    \sum_{n \ll D} \frac{\Theta^{1/2}}{n^{7/6}} \ll \left( \sum_{n \ll D} \frac{1}{n} \right)^{1/2}  \left( \; \mathop{\sum\sum}_{n^2 m \ll K} \frac{|B(n,m)|^2}{(n^2 m)^{2/3}}\right)^{1/2} \ll K^{1/6}.
\end{eqnarray*}
Thus,
\begin{equation*}
    \mathcal{S}_{k,E}(X) \ll \frac{X^{5/3 + \varepsilon} Y K^{1/6}}{\mathcal{Q}^{7/4}}.
\end{equation*}
\end{proof}
\subsection{Error term} \label{error}
In this subsection we give estimates for $\mathcal{S}_k(X)$ corresponding to the case $n^2m X/q^3 \ll X^\varepsilon$ (see subsection \ref{subsection voronoi}), from equation \eqref{voronoi1} we have
\begin{multline} \label{voronoi error}
    \sum_{r\in \mathbb{Z}}\mathcal{A}(r)e\left(\frac{ar}{q}\right) v_u(r) = q\sum_{\pm} \sum_{n\mid q} \sum_{m=1}^\infty \frac{B(n,m)}{nm} S\left(\overline{a},\pm m; \frac{q}{n}\right) G_\pm\left(\frac{n^2m}{q^3}\right) \\ + \frac{1}{2q^2} \Tilde{V}(v_u a,q),
\end{multline}
where $G_\pm$ is defined in equation \eqref{Gpm} and the second term will appear in the case of $d_3$. On substituting equations \eqref{T poisson} and \eqref{voronoi error} into equation \eqref{Sk(X) before summation}, we get
\begin{multline} \label{SkX of yx<}
   \mathcal{S}_k(X) = \frac{X}{\mathcal{Q}}  \sum_{1 \leqslant q \leqslant \mathcal{Q}} \frac{1}{q^2} \sum_{\pm} \sum_{n \mid q} \sum_{m\ll \frac{q^3 X^\varepsilon}{n^2X}} \frac{B(n,m)}{nm} \sum_{n_3 \leqslant Y} \mathsf{a}(n_3)  \\ \times \sum_{m_1,m_2\in \mathbb{Z}} \mathfrak{E_2}(...) \int_{\mathbb{R}} \psi(q, u) U(u) \mathfrak{J}(...) G_\pm\left(\frac{n^2m}{q^3}\right) e\left(\frac{-u n_3^k}{q\mathcal{Q}}\right)  du \\ + \frac{X}{2\mathcal{Q}} \sum_{1 \leqslant q \leqslant \mathcal{Q}} \frac{1}{q^5} \sum_{n_3 \leqslant Y} \mathsf{a}(n_3) \sum_{m_1,m_2\in \mathbb{Z}} \int_{\mathbb{R}} \psi(q,u) U(u) \mathfrak{C}_1(...) \mathfrak{J}(...)du,
\end{multline}
where
$\mathfrak{J}$, $\mathfrak{E_1}$ and $\mathfrak{C}_2$ are defined in equations \eqref{integral poisson}, \eqref{C1}, and \eqref{C2}, respectively. We begin by estimating the contribution of the first term to $\mathcal{S}_k(X)$, denoted by $\mathcal{S}_{k,E}(X)$. The other term will be estimated in the next section. Recall from \eqref{G} and \eqref{Gpm}, we get
\begin{equation*}
    G_\pm(y) = \frac{1}{2\pi i}\int_{(\sigma)} y^{-s} \kappa_\pm(s) \Tilde{v_u}(-s) ds,
\end{equation*}
where $\sigma > -1 + \text{max}\{-\text{Re}(\alpha_1), -\text{Re}(\alpha_2), -\text{Re}(\alpha_3)\}$ with 
\begin{equation*}
    \kappa_\pm(s) = \frac{\pi^{-3s-3/2}}{2} \left( \prod_{j=1}^3 \frac{\Gamma\left(\frac{1+s+\alpha_j}{2}\right)}{\Gamma\left(\frac{-s-\alpha_j}{2}\right)} \mp i \prod_{j=1}^3 \frac{\Gamma\left(\frac{1+s+\alpha_j + 1}{2}\right)}{\Gamma\left(\frac{-s-\alpha_j + 1}{2}\right)}\right).
\end{equation*}
Now, using the second derivative bound given in Lemma \ref{exponential sum}, we get
\begin{equation*}
    \Tilde{v_u}(-s) = X^{-s} \int_0^\infty V(z) e\Bigg(\frac{uzX}{q\mathcal{Q}}\Bigg) z^{-s} \frac{dz}{z} \ll X^{-\sigma} \sqrt{\frac{q\mathcal{Q}}{uX}}.
\end{equation*}
For the gamma factor, we move the contour $\sigma$ to the left up to $\sigma = -5/2$ passing through the poles given by
\begin{equation*}
    \frac{1+\sigma + \text{Re}(\alpha_j) + \ell}{2} = 0 \iff \sigma = -1-\text{Re}(\alpha_j) - \ell.
\end{equation*}
Thus, we get the bound
\begin{equation*}
    G_\pm(y)\ll (yX)^{5/2} \sqrt{\frac{q\mathcal{Q}}{uX}} \int_{-\infty}^\infty |\kappa_\pm(-5/2+i\tau)| d\tau + \sqrt{\frac{q\mathcal{Q}}{uX}} \sum_{\ell = 0,1} \sum_{j=1}^3 (yX)^{1+\ell + \text{Re}(\alpha_j)},
\end{equation*}
the second term above represents the contribution of the poles. Using the Stirling bound $|\kappa_\pm(-5/2 + i\tau)| \ll (1+|\tau|)^{-6}$, we arrive at
\begin{equation*}
    G_\pm(y) \ll \sqrt{\frac{q\mathcal{Q}}{uX}} \left( (yX)^{5/2} + \sum_{\ell = 0,1} \sum_{j=1}^3 (yX)^{1+\ell + \text{Re}(\alpha_j)} \right).
\end{equation*}
 Since $1+ \ell + \text{Re}(\alpha_j) = 1/2 + \beta_j$ for some $\beta_j>0$, and $yX \ll X^\varepsilon$, we have
\begin{equation*}
    G_\pm\Bigg(\frac{n^2m}{q^3}\Bigg) \ll \sqrt{\frac{n^2m \mathcal{Q}}{q^2 u}}.
\end{equation*}
Inserting this bound, together with Weil's bound for Kloosterman sums, and Lemma \ref{ramanujan bound}, we get equation \eqref{voronoi error} is bounded by 
\begin{equation*}
    \frac{q^2\sqrt{\mathcal{Q}}}{\sqrt{Xu}}.
\end{equation*}
Compared with the trivial bound of equation \eqref{voronoi error}, which is $X$. In this case, we are saving $X^{3/2}/(q^2\sqrt{Q})$, which is the same as what we had saved earlier in \eqref{voronoi yx>}. (We have ignored $\sqrt{u}$ in the denominator as the integral over $u$ of $S_{k,E}(X)$ in \eqref{SkX of yx<} will balance this.) Recall from our analysis in subsection \ref{subsection poisson}, we saved $X^2/\mathcal{Q}$ in $n_1$ and $n_2$ sums, and we get square root cancellation in the sum over $a $ modulo $ q$. In the present case, we get a total savings of $X^{7/2}/(q^2 \mathcal{Q})$, over the trivial bound $X^2Y$ in equation \eqref{Sk(X) before summation}, as the remaining integrals do not exhibit oscillations. The total contribution of $n^2m X/q^3 \ll X^\varepsilon$ to $\mathcal{S}_k(X)$ in equation \eqref{SkX of yx<} is bounded by
\begin{equation*}
    \frac{q^2 Y\mathcal{Q}}{X^{3/2}},
\end{equation*}
which is dominated by the bounds obtained in Lemma \ref{zero freq s} and \ref{nonzero freq s}.
\subsection{Final Estimates:}
Finally, substituting bounds from Lemma \ref{zero freq s} and Lemma \ref{nonzero freq s} into equation \eqref{final Sk(X)}, we get
\begin{equation*}
    \mathcal{S}_{k,E}(X) \ll \frac{X^{5/3+ \varepsilon} Y^{1/2} K^{2/3} }{\mathcal{Q}^2} + \frac{X^{5/3 + \varepsilon} Y K^{1/6}}{\mathcal{Q}^{7/4}}, 
\end{equation*}
as $K \asymp X^{1/2}$ and $\mathcal{Q} = X^{1/2}$, we have the following lemma.
\begin{lemma} \label{final SkE}
The error term defined in equation \eqref{error term} is bounded by:
    \begin{equation*} 
    \mathcal{S}_{k,E}(X) \ll \begin{cases}
        X^{7/8 + \varepsilon} Y  & \text{for} \quad k = 3 \\
        X^{1+\varepsilon} Y^{1/2}  &  \text{for} \quad k\geqslant 4.
    \end{cases}
\end{equation*}
\end{lemma}

\section{Main term in Theorem \ref{thm1}}
In this section, we have proved the main term for Theorem \ref{thm1}. Recall from equation \eqref{main term} the main term is given by
\begin{multline} \label{Skm}
    \mathcal{S}_{k,M}(X) = \frac{X}{2\mathcal{Q}} \sum_{1\leqslant q \leqslant \mathcal{Q}} \frac{1}{q^5} \sum_{n_3 \leqslant Y} \mathsf{a}(n_3) \int_{\mathbb{R}} \psi(q,u) U(u)  \\ \times \mathop{\sum\sum}_{m_1,m_2 \in \mathbb{Z}} \mathfrak{C}_1(m_1,m_2,n_3,a;q)\mathfrak{J}(m_1,m_2,u,q) e\left(\frac{-un_3^k}{q\mathcal{Q}}\right) du,
\end{multline}
where $\mathfrak{J}$ is an integral transform defined in equation \eqref{integral poisson} and the character sum $\mathfrak{C}_1$ is defined in equation \eqref{C1} as
\begin{equation*}
    \mathfrak{C}_1(m_1,m_2,n_3,a;q) = \sideset{}{^\star}\sum_{a \; \textrm{mod } q} \Tilde{V}(v_u,a,q) \mathfrak{C}(m_1,m_2,a;q) e\left(\frac{-a n_3^k}{q}\right).
\end{equation*}
Substituting the expression for $\Tilde{\mathcal{V}}$, from equation \eqref{V tilde}, we can write the above sum as
\begin{multline} \label{C1 sum main term}
    \mathfrak{C}_1(...) = \Tilde{v_u}(1) \sum_{n \mid q} n d(n) P_2(n,q) \mathfrak{C}_2(m_1,m_2,n,0,n_3;q) \\
    + \Tilde{v_u}^\prime(1) \sum_{n \mid q} n d(n) P_1(n,q) \mathfrak{C}_2(m_1,m_2,n,0,n_3;q) \\
    + \frac{1}{2} \Tilde{v_u}\dblprime(1) \sum_{n \mid q} n d(n) \mathfrak{C}_2(m_1,m_2,n,0,n_3;q),
\end{multline}
where $\mathfrak{C}_2$ is defined in equation \eqref{C2}. We have the following lemma.
\begin{lemma} \label{bound for C2 main}
    Let $\mathfrak{C}_2$ be defined as in equation \eqref{C2}, then we have
    \begin{equation*}
        \mathfrak{C}_2(m_1,m_2,n,0,n_3;q) \ll q^{3/2}.
    \end{equation*}
\end{lemma}
\begin{proof}
    From equation \eqref{C2 simplified} with $m=0$, we have
    \begin{align*}
       \mathfrak{C}_2(...) &= q  \; \varepsilon_q^2  \sideset{}{^\star}\sum_{a \; \textrm{mod } q} S\left(\overline{a},  0; \frac{q}{n}\right)  e\left(\frac{-\overline{4a}Q^\star (m_1, m_2)}{q} \right) e\left(\frac{-a n_3^k}{q}\right) \\
       &= q \; \varepsilon_q^2  \sideset{}{^\star}\sum_{a \; \textrm{mod } q} \ \sideset{}{^\star}\sum_{\alpha \ \textrm{mod } q/n} e\left(\frac{\overline{a}\alpha}{q/n}\right) e\left(\frac{-\overline{4a}Q^\star (m_1, m_2)}{q} \right) e\left(\frac{-a n_3^k}{q}\right) \\
       &= q \; \varepsilon_q^2  \sideset{}{^\star}\sum_{a \; \textrm{mod } q} \sum_{d \mid (\overline{a}, q/n)} d \mu\left(\frac{q}{nd}\right) e\left(\frac{-\overline{4a}Q^\star (m_1, m_2)}{q} \right) e\left(\frac{-a n_3^k}{q}\right) \\
       &= q \; \varepsilon_q^2 \mu\left(\frac{q}{n}\right) S(-n_3^k,-\overline{4}Q^*(m_1,m_2);q).
    \end{align*}
Using Weil's bound for Kloosterman sums, we will get our desired bound.
\end{proof}

We can write the expression for $\mathcal{S}_{k,M}(X)$ given in equation \eqref{Skm} as the sum of $\mathcal{S}^0_{k,M}(X)$ and $\mathcal{S}^{\neq0}_{k,M}(X)$, where $\mathcal{S}^0_{k,M}(X)$ is the part of the sum with $m_1=m_2=0$ (which will contribute to the main term) and $\mathcal{S}^{\neq0}_{k,M}(X)$ is the remaining part of $\mathcal{S}_{k,M}(X)$.
Further, we denote the right-hand side of the equation \eqref{C1 sum main term} by $\mathscr{S}_j$ for $j=0,1,2$.
Recall from equations \eqref{P1} and \eqref{P2}, we have
\begin{equation} \label{Pj bound}
    P_j(n,q) \ll (\log(n+2)(q+2))^j, \ \ \ \ j=1,2.
\end{equation}
Consider
\begin{multline*}
    \mathcal{S}^{\neq0}_{k,M}(X) = \frac{X}{2\mathcal{Q}} \sum_{1\leqslant q \leqslant \mathcal{Q}} \frac{1}{q^5} \sum_{n_3 \leqslant Y} \mathsf{a}(n_3) \int_{\mathbb{R}} \psi(q,u) U(u)  \\ \times \mathop{\sum\sum}_{m_1,m_2 \in \mathbb{Z}\setminus\{0\}} \mathfrak{C}_1(m_1,m_2,n_3,a;q)\mathfrak{J}(m_1,m_2,u,q) e\left(\frac{-un_3^k}{q\mathcal{Q}}\right) du.
\end{multline*}
The above integral is negligibly small unless $m_1, m_2\ll X^\varepsilon$. To obtain a bound for the sum above, we need to bound the integrals. This can be achieved similarly to the approach outlined earlier in subsection \ref{subsection integral simplify}. Initially, using the $u$-integral, we obtain a condition over $x, v_1$ and $v_2$. Subsequently, by applying the first derivative bound, as given in Lemma \ref{exponential sum}, on the $v_1$ or $v_2$ integral, we find that the above integral is bounded by $\ll q^2/X$ or $q^2/\mathcal{Q}^2$. Now, using this bound, Lemma \ref{bound for C2 main} and equation \eqref{Pj bound}, we obtain
\begin{align} \label{Skm non zero m1m2}
 \mathcal{S}^{\neq0}_{k,M}(X)   &\ll \frac{X^2Y}{\mathcal{Q}} \sum_{1\leqslant q \leqslant \mathcal{Q}} \frac{1}{q^5} \mathop{\sum\sum}_{m_1,m_2 \ll X^\varepsilon } \sum_{n \mid q} n d(n) q^{3/2} \frac{q^2}{X} \notag \\
    &\ll \frac{X^{1+\varepsilon}Y}{\mathcal{Q}} \sum_{1\leqslant q \leqslant \mathcal{Q}} \frac{1}{q^{1/2}} \ll \frac{X^{1+\varepsilon} \mathcal{Q}^{1/2}Y}{\mathcal{Q}} = X^{3/4+\varepsilon} Y.
\end{align}
Thus, we have
\begin{equation} \label{Skm split}
    \mathcal{S}_{k,M}(X) = \mathcal{S}^{0}_{k,M}(X) + O(X^{3/4+\varepsilon} Y).
\end{equation}
The remaining section focuses on simplifying $\mathcal{S}^{0}_{k,M}(X)$. We will be using the bound $q^2/X$, which can be obtained by the condition coming from $u$ integral and another one coming from $x$ integral. For $j=0,1,2$, replacing $\Tilde{v_u}^{(j)}(1)$ by
\begin{equation} \label{phi bj}
    \phi^{b,j}(u) = \int_{X/2}^{3X} e\left(\frac{ux}{q\mathcal{Q}}\right)(\log{x})^j dx,
\end{equation}
we need to estimate the remainder terms from 
\begin{equation*}
    \phi^{\sharp,j}(u) = \Tilde{v_u}^{(j)}(1) - \phi^{b,j}(u).
\end{equation*}
Write correspondingly
\begin{equation*}
    \mathscr{S^{\sharp}}_{j} = \mathscr{S}_{j} - \mathscr{S}_{j}^b, \ \ \ j=0,1,2.
\end{equation*}
Notice that
\begin{equation*}
     \phi^{\sharp,j}(u) = \int_{X/2}^{3 X} \Big( V\left(\frac{x}{X}\right) -1\Big) e\left(\frac{ux}{q \mathcal{Q}}\right)(\log{x})^j dx \ll X (\log{X})^j,
\end{equation*}
and by using integration by parts we get the restriction $|u| \ll X^{\varepsilon}q/\mathcal{Q}.$
Consider
\begin{multline*}
   \mathfrak{A}_0^\sharp := \frac{X}{2\mathcal{Q}} \sum_{1 \leqslant q \leqslant \mathcal{Q}} \frac{1}{q^5}  \sum_{n_3 \leqslant Y} \mathsf{a}(n_3) \\ \times \int_{\mathbb{R}} \psi(q, u)U(u)  \mathscr{S}_{0}^{\sharp}(0,0,n_3,a,q)  \mathfrak{J}(0,0,u,q) e\left(\frac{-u n_3^k}{q\mathcal{Q}}\right)du
\end{multline*}
\begin{multline*}
   \hspace{0.45cm} = \frac{X}{2\mathcal{Q}} \sum_{1 \leqslant q \leqslant \mathcal{Q}} \frac{1}{q^5}  \sum_{n_3 \leqslant Y} \mathsf{a}(n_3)  \sum_{n\mid q}n d(n) P_2(n,q) \mathfrak{C}_2(0,0,n,0,n_3;q)\\ \times \int_{\mathbb{R}} \phi^{\sharp,0}(u) \psi(q, u)U(u) \mathfrak{J}(0,0,u,q)e\left(\frac{-u n_3^k}{q\mathcal{Q}}\right) du.
\end{multline*}
Using Lemma \ref{bound for C2 main}, bound for $\phi^{\sharp,0}$, and $\mathfrak{J}$, we obtain
\begin{align} \label{A0hash}
    \mathfrak{A}_0^\sharp &\ll X^{3/2} \sum_{1 \leqslant q \leqslant \mathcal{Q}} \frac{1}{q^5} \sum_{n_3 \leqslant Y} \mathsf{a}(n_3) \sum_{n\mid q} n d(n) |P_2(n,q)| q^{3/2} \frac{q^2}{X}  \notag \\ &\ll X^{1/2+\varepsilon} Y \sum_{1 \leqslant q \leqslant \mathcal{Q}} \frac{1}{q^{1/2}} \ll X^{3/4+\varepsilon} Y.
\end{align}
Similarly, we can have
\begin{equation} \label{A1hash A2hash}
    \mathfrak{A}_1^\sharp \ll X^{3/4+\varepsilon} Y \ \textrm{and} \ \mathfrak{A}_2^\sharp \ll X^{3/4+\varepsilon} Y.
\end{equation}
Next, 
\begin{multline*}
   \mathfrak{A}_0^b:= \frac{X}{2\mathcal{Q}} \sum_{1 \leqslant q \leqslant \mathcal{Q}} \frac{1}{q^5}  \sum_{n_3 \leqslant Y}  \mathsf{a}(n_3) \\ \times \int_{\mathbb{R}} \psi(q, u)U(u) \mathscr{S}_{0}^{b}(0,0,n_3,a,q) \mathfrak{J}(0,0,u,q)e\left(\frac{-u n_3^k}{q\mathcal{Q}}\right) du
\end{multline*}
\begin{multline*}
   \hspace{0.45cm} = \frac{X}{2\mathcal{Q}} \sum_{1 \leqslant q \leqslant \mathcal{Q}} \frac{1}{q^5}  \sum_{n_3 \leqslant Y} \mathsf{a}(n_3) \sum_{n\mid q}n d(n) P_2(n,q) \mathfrak{C}_2(0,0,n,0,n_3;q)\\ \times \int_{\mathbb{R}} \phi^{b,0}(u) \psi(q, u)U(u) \mathfrak{J}(0,0,u,q) e\left(\frac{-u n_3^k}{q\mathcal{Q}}\right)du.
\end{multline*} 
Now, as $n_3^k \leqslant X$, we can write $n_3^k = X - \ell$ where $0 \leqslant \ell \leqslant X-1$,
\begin{equation*}
    e\left(\frac{-u n_3^k}{q\mathcal{Q}}\right) = e\left(\frac{-u X +u \ell}{q\mathcal{Q}}\right)= e\left(\frac{-u X^{1/2}}{q}\right) + O(1),
\end{equation*}
with this, we arrive at
\begin{multline*} 
  \mathfrak{A}_0^b  = \frac{X}{2\mathcal{Q}} \sum_{1 \leqslant q \leqslant \mathcal{Q}} \frac{1}{q^5} \sum_{n\mid q}n d(n) P_2(n,q) \ \sideset{}{^\star}\sum_{a \; \textrm{mod } q} S\left(\overline{a}, 0; \frac{q}{n}\right) \mathfrak{C}(0,0,a;q) \\ \times \sum_{n_3 \leqslant Y} \mathsf{a}(n_3) e\left(\frac{-a n_3^k}{q}\right) \int_{\mathbb{R}} \phi^{b,0}(u) \psi(q,u) U(u)  \mathfrak{J}(0,0,u,q) e\left(\frac{-u X^{1/2}}{q}\right) du\\ +O\Big(X^{3/4+\varepsilon} Y\Big)
\end{multline*}
\begin{multline} \label{equation with a(n)}
= \frac{X}{2\mathcal{Q}} \sum_{1 \leqslant q \leqslant \mathcal{Q}} \frac{1}{q^5} \sum_{n\mid q}n d(n) P_2(n,q) \ \sideset{}{^\star}\sum_{a \; \textrm{mod } q} S\left(\overline{a}, 0; \frac{q}{n}\right) \mathfrak{C}(0,0,a;q) \\ \times \sum_{\alpha \ \textrm{mod } q}  \sum_{\substack{n_3 \leqslant Y \\ n_3 \equiv \alpha \ \textrm{mod } q}} \mathsf{a}(n_3)  e\left(\frac{-a \alpha^k}{q}\right) \\ \times \int_{\mathbb{R}} \phi^{b,0}(u) \psi(q,u) U(u)  \mathfrak{J}(0,0,u,q) e\left(\frac{-u X^{1/2}}{q}\right) du + O\Big(X^{3/4+\varepsilon} Y\Big).
\end{multline}

\noindent $\bullet$ \textit{When $\mathsf{a}(n)= \Lambda(n)$}. We split the sum over $q$ into two parts over the ranges $[1,Y^{1/2}]$ and $(Y^{1/2}, \mathcal{Q}]$. For the first range, we use the Bombieri-Vinogradov theorem and other can be treated trivially. Label them as $\mathcal{B}_i$ for $i=1,2$. Consider
\begin{multline*}
    \mathcal{B}_2 = \frac{X}{2\mathcal{Q}} \sum_{Y^{1/2} < q \leqslant \mathcal{Q}} \frac{1}{q^5} \sum_{n\mid q}n d(n) P_2(n,q) \ \sideset{}{^\star}\sum_{a \; \textrm{mod } q} S\left(\overline{a}, 0; \frac{q}{n}\right) \mathfrak{C}(0,0,a;q) \\ \times \sum_{n_3 \leqslant Y} \Lambda(n_3) e\left(\frac{-a n_3^k}{q}\right)  \int_{\mathbb{R}} \phi^{b,0}(u) \psi(q,u) U(u)  \mathfrak{J}(0,0,u,q) e\left(\frac{-u X^{1/2}}{q}\right)du
\end{multline*}
\begin{equation} \label{B2}
     \ll \frac{X}{\mathcal{Q}} \sum_{Y^{1/2} < q \leqslant \mathcal{Q}} \frac{1}{q^5} \sum_{n\mid q}n d(n) |P_2(n,q)| \sum_{n_3 \leqslant Y} \Lambda(n_3)\cdot q^{3/2} \cdot q^2 \ll X^{3/4 +\varepsilon}Y.
\end{equation}
and
\begin{multline} \label{B1}
\mathcal{B}_1 = \frac{X}{2\mathcal{Q}} \sum_{1 \leqslant q \leqslant Y^{1/2}} \frac{1}{q^5} \sum_{n\mid q}n d(n) P_2(n,q) \ \sideset{}{^\star}\sum_{a \; \textrm{mod } q} S\left(\overline{a}, 0; \frac{q}{n}\right) \mathfrak{C}(0,0,a;q) \\ \times \sum_{\alpha \ \textrm{mod } q} \Bigg( \sum_{\substack{n_3 \leqslant Y \\ n_3 \equiv \alpha \ \textrm{mod } q}} \Lambda(n_3) - \frac{Y}{\phi(q)}  +  \frac{Y}{\phi(q)} \Bigg)e\left(\frac{-a \alpha^k}{q}\right) \\ \times \int_{\mathbb{R}} \phi^{b,0}(u) \psi(q,u) U(u) \mathfrak{J}(0,0,u,q) e\left(\frac{-u X^{1/2}}{q}\right) du.
\end{multline}
Now, we write $\mathcal{B}_1 = \mathcal{B}_{11} + \mathcal{B}_{12}$, where 
\begin{multline*}
    \mathcal{B}_{11} = \frac{X}{2\mathcal{Q}} \sum_{1 \leqslant q \leqslant Y^{1/2}} \frac{1}{q^5} \sum_{n\mid q}n d(n) P_2(n,q) \ \sideset{}{^\star}\sum_{a \; \textrm{mod } q} S\left(\overline{a}, 0; \frac{q}{n}\right) \mathfrak{C}(0,0,a;q) \\ \times \sum_{\alpha \ \textrm{mod } q} \Bigg( \sum_{\substack{n_3 \leqslant Y \\ n_3 \equiv \alpha \ \textrm{mod } q}} \Lambda(n_3) - \frac{Y}{\phi(q)} \Bigg)e\left(\frac{-a \alpha^k}{q}\right) \\ \times \int_{\mathbb{R}} \phi^{b,0}(u) \psi(q,u) U(u) \mathfrak{J}(0,0,u,q)e\left(\frac{-u X^{1/2}}{q}\right) du.
\end{multline*}
Next, we split the sum over $q$ into dyadic intervals of size $q \sim C \ll Y^{1/2}$.
Using the Bombieri-Vinogradov theorem for each such interval, we can bound the above expression by
\begin{eqnarray*}
    \mathcal{B}_{11} &\ll& \sup_{C \ll Y^{1/2}} \frac{X}{\mathcal{Q}} \sum_{q \sim C} \frac{1}{q^5}\sum_{n \mid q} nd(n) |P_2(n,q)| \ \sideset{}{^\star}\sum_{a \; \textrm{mod } q} S\left(\overline{a}, 0; \frac{q}{n}\right) \mathfrak{C}(0,0,a;q) \\
    && \hspace{3.8cm}\times \sum_{\alpha \textrm{ mod } q} \Bigg( \sum_{\substack{n_3 \leqslant Y \\ n_3 \equiv \alpha \ \textrm{mod } q}} \Lambda(n_3) - \frac{Y}{\phi(q)} \Bigg)e\left(\frac{-a \alpha^k}{q}\right) q^2 \\
    &\ll& \sup_{C\ll Y^{1/2}} \frac{X^{1+\varepsilon}Y^{1/2}C^{3/2}\log{C}} {\mathcal{Q}} \ll X^{1/2+\varepsilon}Y^{5/4}.
\end{eqnarray*}
The remaining sum $\mathcal{B}_{12}$ will contribute to the main term. Recall the properties of the test function $\psi(q,u)$ (refer to Lemma \ref{delta}); for small $q \ll \mathcal{Q}^{1-\varepsilon}$, it can be replaced by $1$ with only a negligible error term.
\begin{multline*}
    \mathcal{B}_{12} = \frac{X^{1/2} Y}{2} \sum_{q=1}^\infty \frac{1}{q^5\phi(q)} \sum_{n\mid q}n d(n) P_2(n,q) \sideset{}{^\star}\sum_{a \ \textrm{mod } q} S\left(\overline{a}, 0; \frac{q}{n}\right) \mathfrak{C}(0,0,a;q) \\ \times \sum_{\alpha \ \textrm{mod } q} e\left(\frac{-a \alpha^k}{q}\right) \int_{\mathbb{R}} \phi^{b,0}(u)  U(u) \mathfrak{J}(0,0,u,q) e\left(\frac{-u X^{1/2}}{q}\right) du \\+ O\Big(X^{1/2+\varepsilon}Y^{3/4}\Big),
\end{multline*}
as
\begin{multline*}
    \frac{X^{1/2} Y}{2} \sum_{q > Y^{1/2}}\frac{1}{q^5\phi(q)} \sum_{n\mid q}n d(n) P_2(n,q) \sideset{}{^\star}\sum_{a \ \textrm{mod } q} S\left(\overline{a}, 0; \frac{q}{n}\right) \mathfrak{C}(0,0,a;q) \\ \times \sum_{\alpha \ \textrm{mod } q} e\left(\frac{-a \alpha^k}{q}\right) \int_{\mathbb{R}} \phi^{b,0}(u) \psi(q,u) U(u) \mathfrak{J}(0,0,u,q) e\left(\frac{-u X^{1/2}}{q}\right)du
\end{multline*}
\begin{align*}
    &\ll X^{3/2} Y \sum_{q > Y^{1/2}}\frac{1}{q^5\phi(q)} \sum_{n\mid q}n d(n) |P_2(n,q)| q^{3/2}  \frac{q^2}{X} \\ 
    &\ll X^{1/2+\varepsilon}Y \sum_{q > Y^{1/2}}\frac{1}{q^{1/2}\phi(q)} \ll X^{1/2+\varepsilon}Y^{3/4}.
\end{align*}
Finally, from equations \eqref{equation with a(n)}, \eqref{B2} and \eqref{B1}, we obtain
\begin{multline} \label{A0b}
    \mathfrak{A}_0^b = \frac{X^{1/2} Y}{2} \sum_{q=1}^\infty \frac{1}{q^5\phi(q)} \sum_{n\mid q}n d(n) P_2(n,q) \sideset{}{^\star}\sum_{a \ \textrm{mod } q} S\left(\overline{a}, 0; \frac{q}{n}\right) \mathfrak{C}(0,0,a;q) \\ \times \sum_{\alpha \ \textrm{mod } q} e\left(\frac{-a \alpha^k}{q}\right) \int_{\mathbb{R}} \phi^{b,0}(u)  U(u) \mathfrak{J}(0,0,u,q) e\left(\frac{-u X^{1/2}}{q}\right) du \\ + O\Big( X^{3/4 + \varepsilon} Y\Big).
\end{multline}
Similarly, we can get
\begin{multline} \label{A1b}
    \mathfrak{A}_1^b = \frac{X^{1/2} Y}{2} \sum_{q=1}^\infty \frac{1}{q^5\phi(q)} \sum_{n\mid q}n d(n) P_1(n,q) \sideset{}{^\star}\sum_{a \ \textrm{mod } q} S\left(\overline{a}, 0; \frac{q}{n}\right) \mathfrak{C}(0,0,a;q) \\ \times \sum_{\alpha \ \textrm{mod } q} e\left(\frac{-a \alpha^k}{q}\right) \int_{\mathbb{R}} \phi^{b,1}(u)  U(u) \mathfrak{J}(0,0,u,q) e\left(\frac{-u X^{1/2}}{q}\right)du \\ +  O\Big( X^{3/4 + \varepsilon} Y\Big),
\end{multline}
and
\begin{multline} \label{A2b}
    \mathfrak{A}_2^b = \frac{X^{1/2} Y}{4} \sum_{q=1}^\infty \frac{1}{q^5\phi(q)} \sum_{n\mid q}n d(n) \sideset{}{^\star}\sum_{a \ \textrm{mod } q} S\left(\overline{a}, 0; \frac{q}{n}\right) \mathfrak{C}(0,0,a;q) \\ \times \sum_{\alpha \ \textrm{mod } q} e\left(\frac{-a \alpha^k}{q}\right) \int_{\mathbb{R}} \phi^{b,2}(u)  U(u) \mathfrak{J}(0,0,u,q) e\left(\frac{-u X^{1/2}}{q}\right)du \\ +  O\Big( X^{3/4 + \varepsilon} Y\Big).
\end{multline}
To simplify the main term further, we need to analyze the integrals. Therefore, let us consider 
\begin{equation*}
   \mathcal{G}_j:= \int_{\mathbb{R}} \phi^{b,j}(u) U(u) \mathfrak{J}(0,0,u,q)e\left(\frac{-u X^{1/2}}{q}\right) du.  
\end{equation*}
Substituting the expression for $\phi^{b,j}$ from equation \eqref{phi bj}, the integral $\mathfrak{J}$ from equation \eqref{integral poisson}, and applying the change of variable $x \rightarrow xX$, we obtain
\begin{multline*}
  \mathcal{G}_j  = X \int_{\mathbb{R}} \Bigg( \int_{1/2}^{3} e\left(\frac{uxX}{q\mathcal{Q}}\right)(\log{xX})^j dx \Bigg) U(u) e\left(\frac{-u X^{1/2}}{q}\right) \\ \times \iint_{\mathbb{R}^2} W_1(v_1) W_2(v_2) e\left(\frac{-uQ(v_1X^{1/2}, v_2X^{1/2})}{q\mathcal{Q}} \right) dv_1 dv_2 du
\end{multline*}
\begin{multline*}
   \hspace{0.3cm} = q X \int_{\mathbb{R}} \Bigg( \int_{1/2}^{3} e(uxX^{1/2})(\log{xX})^j dx \Bigg)  U(uq) e(-u X^{1/2}) \\ \times \iint_{\mathbb{R}^2} W_1(v_1) W_2(v_2) e\left(\frac{-uQ(v_1X^{1/2}, v_2X^{1/2})}{\mathcal{Q}} \right) dv_1 dv_2 du
\end{multline*}
\begin{multline*}
   \hspace{0.3cm} = q X^{1/2} \int_{\mathbb{R}} \Bigg( \int_{1/2}^{3} e(ux)(\log{xX})^j dx \Bigg) U\left(\frac{uq}{X^{1/2}}\right) e(-u) \\ \times \iint_{\mathbb{R}^2} W_1(v_1) W_2(v_2) e\left(\frac{-uQ(v_1X^{1/2}, v_2X^{1/2})}{X} \right) dv_1 dv_2 du.
\end{multline*}
Using
\begin{align*}
    e\left(\frac{-uQ(v_1X^{1/2}, v_2X^{1/2})}{X} \right) &= e\left(-uQ_h(v_1,v_2) \right) e\left(-u\left(\frac{Dv_1}{X^{1/2}} + \frac{Ev_2}{X^{1/2}} + \frac{F}{X}\right)\right) \\
    &= e\left(-uQ_h(v_1,v_2) \right) + O\Bigg(\frac{1}{X^{1/2}}\Bigg),
\end{align*}
where $Q_h$ represents the homogeneous part of the polynomial $Q$. Furthermore, the bump function $U$ can be substituted with $1$ up to some negligible error. Thus, the above equation, with some negligible error, becomes:
\begin{multline*}
     q X^{1/2}\int_{\mathbb{R}}  \Bigg( \int_{1/2}^{3} e(ux)(\log{xX})^j dx \Bigg) \\ \times \iint_{\mathbb{R}^2} W_1(v_1) W_2(v_2) e\left(-u(Q_h(v_1, v_2)+1) \right) dv_1 dv_2 du.
\end{multline*}
Thus, for $j=0,1,2$, we have
\begin{equation} \label{Gj}
   \mathcal{G}_j = q X^{1/2} \sum_{\ell=0}^j \binom{j}{\ell} (\log{X})^{j-\ell} \mathcal{J}_\ell,
\end{equation}
where
\begin{multline} \label{J}
   \mathcal{J}_\ell:=  \int_{\mathbb{R}}  \Bigg( \int_{1/2}^{3} e(ux)(\log{x})^\ell dx \Bigg) \\ \times \iint_{\mathbb{R}^2} W_1(v_1) W_2(v_2) e\left(-u(Q_h(v_1, v_2)+1) \right) dv_1 dv_2 du.
\end{multline}
Substituting this in the equations \eqref{A0b}-\eqref{A2b}, and using equations \eqref{A0hash} and \eqref{A1hash A2hash}, we obtain
\begin{equation*}
     \mathfrak{A}_0 = \frac{X Y}{2} \mathcal{J}_0 \mathcal{S}_2 + O\Big( X^{3/4 + \varepsilon} Y\Big),
\end{equation*}
where
\begin{equation} \label{S2}
    \mathcal{S}_2 = \sum_{q=1}^\infty \frac{1}{q^4\phi(q)} \sum_{n\mid q}n d(n) P_2(n,q) \sideset{}{^\star}\sum_{a \ \textrm{mod } q} S\left(\overline{a}, 0; \frac{q}{n}\right) \mathfrak{C}(0,0,a;q) \sum_{\alpha \ \textrm{mod } q} e\left(\frac{-a \alpha^k}{q}\right).
\end{equation}
Next,
\begin{equation*}
    \mathfrak{A}_1 = \frac{XY }{2} (\log{X}) \mathcal{J}_0 \mathcal{S}_1 + \frac{XY}{2} \mathcal{J}_1 \mathcal{S}_1 +O\Big(X^{3/4+\varepsilon} Y\Big),
\end{equation*}
where
\begin{equation} \label{S1}
    \mathcal{S}_1 = \sum_{q=1}^\infty \frac{1}{q^4\phi(q)} \sum_{n\mid q}n d(n) P_1(n,q) \sideset{}{^\star}\sum_{a \ \textrm{mod } q} S\left(\overline{a}, 0; \frac{q}{n}\right) \mathfrak{C}(0,0,a;q) \sum_{\alpha \ \textrm{mod } q} e\left(\frac{-a \alpha^k}{q}\right),
\end{equation}
and
\begin{equation*}
    \mathfrak{A}_2 = \frac{XY}{4} (\log{X})^2 \mathcal{J}_0 \mathcal{S}_0 + \frac{XY}{2} (\log{X}) \mathcal{J}_1 \mathcal{S}_0 + \frac{XY}{4} \mathcal{J}_2 \mathcal{S}_0 +O\Big(X^{3/4+\varepsilon} Y\Big),
\end{equation*}
with
\begin{equation} \label{S0}
    \mathcal{S}_0 = \sum_{q=1}^\infty \frac{1}{q^4\phi(q)} \sum_{n\mid q}n d(n) \sideset{}{^\star}\sum_{a \ \textrm{mod } q} S\left(\overline{a}, 0; \frac{q}{n}\right) \mathfrak{C}(0,0,a;q) \sum_{\alpha \ \textrm{mod } q} e\left(\frac{-a \alpha^k}{q}\right).
\end{equation}
Summarizing the above analysis and using equation \eqref{Skm split}. We have the following lemma.
\begin{lemma} \label{main term von}
    The main term in Theorem \ref{thm1}, for the case of $\Lambda(n)$, denoted by $\mathcal{S}_{k,M}(X)$ is given as:
    \begin{multline*} 
   \mathcal{S}_{k,M}(X)=  \frac{XY}{4} (\log{X})^2 \mathcal{J}_0 \mathcal{S}_0 + \frac{XY}{2} (\log{X}) \left( \mathcal{J}_0\mathcal{S}_1 + \mathcal{J}_1\mathcal{S}_0\right) \\+ \frac{XY}{2}\left( \mathcal{J}_0 \mathcal{S}_2+ \mathcal{J}_1\mathcal{S}_1 + \frac{1}{2}\mathcal{J}_2\mathcal{S}_0\right) +O\Big(X^{3/4+\varepsilon} Y\Big),
\end{multline*}
where for $j=0,1,2$, $\mathcal{S}_j$ and $\mathcal{J}_j$ are defined in equations \eqref{S2}-\eqref{S0} and \eqref{J}, respectively.
\end{lemma}
\noindent $\bullet$ \textit{When $\mathsf{a}(n)= 1$}. We split the sum over $q$ in equation \eqref{equation with a(n)} into two parts over the ranges $[1,Y)$ and $[Y,\mathcal{Q}]$. The first range will contribute to the main term, other will go in error. Label them as $\mathcal{B}_i$ for $i=3,4$. Consider
\begin{multline*}
  \mathcal{B}_3 = \frac{XY}{2\mathcal{Q}} \sum_{1 \leqslant q < Y} \frac{1}{q^6} \sum_{n\mid q}n d(n) P_2(n,q) \sum_{\alpha \ \textrm{mod } q} \ \sideset{}{^\star}\sum_{a \; \textrm{mod } q} S\left(\overline{a}, 0; \frac{q}{n}\right) \mathfrak{C}(0,0,a;q) e\left(\frac{-a \alpha^k}{q}\right) \\ \times \int_{\mathbb{R}} \phi^{b,0}(u) \psi(q,u) U(u) \mathfrak{J}(0,0,u,q) e\left(\frac{-u X^{1/2}}{q}\right)du + O\left(X^{1/2+\varepsilon}Y^{3/2}\right),
\end{multline*} as
\begin{equation*}
    \sum_{\substack{n_3\leqslant Y \\ n_3 \equiv \alpha \ \textrm{mod } q}} 1 = \frac{Y}{q} + O(1).
\end{equation*}
Next,
\begin{multline*}
    \mathcal{B}_4 = \frac{X}{2\mathcal{Q}} \sum_{Y \leqslant q \leqslant \mathcal{Q}} \frac{1}{q^5} \sum_{n\mid q}n d(n) P_2(n,q) \sum_{n_3 \leqslant Y} \ \sideset{}{^\star}\sum_{a \; \textrm{mod} \; q} S\left(\overline{a}, 0; \frac{q}{n}\right) \mathfrak{C}(0,0,a;q) e\left(\frac{-a n_3^k}{q}\right) \\ \times \int_{\mathbb{R}} \phi^{b,0}(u) \psi(q,u) U(u)  \mathfrak{J}(0,0,u,q) e\left(\frac{-u X^{1/2}}{q}\right) du
\end{multline*}
\begin{equation*}
    \ll \frac{XY}{\mathcal{Q}} \sum_{Y \leqslant q \leqslant \mathcal{Q}} \frac{1}{q^5} \sum_{n\mid q}n d(n) |P_2(n,q)| q^{3/2} q^2 \ll X^{3/4+\varepsilon}Y.
\end{equation*}
Also,
\begin{multline*}
    \frac{X}{2\mathcal{Q}} \sum_{ q > \mathcal{Q}} \frac{1}{q^5} \sum_{n\mid q}n d(n) P_2(n,q) \sum_{n_3 \leqslant Y} \ \sideset{}{^\star}\sum_{a \; \textrm{mod } q} S\left(\overline{a}, 0; \frac{q}{n}\right) \mathfrak{C}(0,0,a;q) e\left(\frac{-a n_3^k}{q}\right) \\ \times \int_{\mathbb{R}} \phi^{b,0}(u) \psi(q,u) U(u) \mathfrak{J}(0,0,u,q) e\left(\frac{-u X^{1/2}}{q}\right)du
\end{multline*}
\begin{equation*}
    \ll X^{3/2}Y \sum_{ q > \mathcal{Q}} \frac{1}{q^5} \sum_{n\mid q}n d(n) |P_2(n,q)|\cdot q^{3/2} \ll X^{3/4+\varepsilon}Y. 
\end{equation*}
Here, we have treated the integral trivially. Hence, we obtain
\begin{multline} \label{A0b1}
    \mathfrak{A}_0^b = \frac{X^{1/2}Y}{2} \sum_{q =1}^\infty \frac{1}{q^6} \sum_{n\mid q}n d(n) P_2(n,q) \sum_{\alpha \ \textrm{mod } q} \ \sideset{}{^\star}\sum_{a \; \textrm{mod } q}  S\left(\overline{a}, 0; \frac{q}{n}\right) \mathfrak{C}(0,0,a;q) e\left(\frac{-a \alpha^k}{q}\right) \\ \times \int_{\mathbb{R}} \phi^{b,0}(u) U(u) \mathfrak{J}(0,0,u,q) e\left(\frac{-u X^{1/2}}{q}\right)du +O(X^{3/4+\varepsilon}Y).
\end{multline}
Analogously, we can have
\begin{multline} \label{A1b1}
    \mathfrak{A}_1^b = \frac{X^{1/2} Y}{2} \sum_{q=1}^\infty \frac{1}{q^6} \sum_{n\mid q}n d(n) P_1(n,q) \sum_{\alpha \ \textrm{mod } q} \ \sideset{}{^\star}\sum_{a \ \textrm{mod } q} S\left(\overline{a}, 0; \frac{q}{n}\right) \mathfrak{C}(0,0,a;q)  e\left(\frac{-a \alpha^k}{q}\right) \\ \times \int_{\mathbb{R}} \phi^{b,1}(u) U(u)  \mathfrak{J}(0,0,u,q) e\left(\frac{-u X^{1/2}}{q}\right) du +O(X^{3/4+\varepsilon}Y).
\end{multline}
and
\begin{multline} \label{A2b1}
    \mathfrak{A}_2^b = \frac{X^{1/2} Y}{2} \sum_{q=1}^\infty \frac{1}{q^6} \sum_{n\mid q}n d(n) \sum_{\alpha \ \textrm{mod } q} \ \sideset{}{^\star}\sum_{a \ \textrm{mod } q} S\left(\overline{a}, 0; \frac{q}{n}\right) \mathfrak{C}(0,0,a;q) e\left(\frac{-a \alpha^k}{q}\right) \\ \times \int_{\mathbb{R}} \phi^{b,2}(u) U(u)  \mathfrak{J}(0,0,u,q) e\left(\frac{-u X^{1/2}}{q}\right) du + O\Big( X^{3/4 + \varepsilon} Y\Big).
\end{multline}
Using equation \eqref{Gj}, we can simplify the integral expressions of the above equations. After that from equations \eqref{A0hash}, and \eqref{A1hash A2hash}, we obtain
\begin{equation*}
     \mathfrak{A}_0 = \frac{X Y}{2} \mathcal{J}_0 \mathcal{C}_2 + O\Big( X^{3/4 + \varepsilon} Y\Big),
\end{equation*}
where
\begin{equation} \label{C2main}
    \mathcal{C}_2 = \sum_{q=1}^\infty \frac{1}{q^6} \sum_{n\mid q}n d(n) P_2(n,q) \sum_{\alpha \ \textrm{mod } q} \ \sideset{}{^\star}\sum_{a \ \textrm{mod } q} S\left(\overline{a}, 0; \frac{q}{n}\right) \mathfrak{C}(0,0,a;q)  e\left(\frac{-a \alpha^k}{q}\right).
\end{equation}
Further,
\begin{equation*}
    \mathfrak{A}_1 = \frac{XY}{2} ( \log{X}) \mathcal{J}_0 \mathcal{C}_1 + \frac{XY}{2} \mathcal{J}_1 \mathcal{C}_1 +O\Big(X^{3/4+\varepsilon} Y\Big),
\end{equation*}
where
\begin{equation} \label{C1main}
    \mathcal{C}_1 = \sum_{q=1}^\infty \frac{1}{q^6} \sum_{n\mid q}n d(n) P_1(n,q) \sum_{\alpha \ \textrm{mod } q} \ \sideset{}{^\star}\sum_{a \ \textrm{mod } q} S\left(\overline{a}, 0; \frac{q}{n}\right) \mathfrak{C}(0,0,a;q)  e\left(\frac{-a \alpha^k}{q}\right),
\end{equation}
and
\begin{equation*}
    \mathfrak{A}_2 = \frac{XY}{4} (\log{X})^2 \mathcal{J}_0 \mathcal{C}_0 + \frac{XY}{2} (\log{X}) \mathcal{J}_1 \mathcal{C}_0 + \frac{XY}{4} \mathcal{J}_2 \mathcal{C}_0 +O\Big(X^{3/4+\varepsilon} Y\Big),
\end{equation*}
with
\begin{equation} \label{C0main}
    \mathcal{C}_0 = \sum_{q=1}^\infty \frac{1}{q^6} \sum_{n\mid q}n d(n) \sum_{\alpha \ \textrm{mod } q} \ \sideset{}{^\star}\sum_{a \ \textrm{mod } q} S\left(\overline{a}, 0; \frac{q}{n}\right) \mathfrak{C}(0,0,a;q)  e\left(\frac{-a \alpha^k}{q}\right).
\end{equation}
Using equation \eqref{Skm split}, the following lemma records the main term for the case of the identity function.
\begin{lemma} \label{main term id}
    The main term in Theorem \ref{thm1}, for the case of identity function, denoted by $\mathcal{S}_{k,M}(X)$ is given as:
    \begin{multline*}
       \mathcal{S}_{k,M}(X)= \frac{XY}{4} (\log{X})^2 \mathcal{J}_0 \mathcal{C}_0 + \frac{XY}{2} (\log{X}) \left( \mathcal{J}_0\mathcal{C}_1 + \mathcal{J}_1\mathcal{C}_0\right) \\+ \frac{XY}{2}\left( \mathcal{J}_0 \mathcal{C}_2+ \mathcal{J}_1\mathcal{C}_1 + \frac{1}{2}\mathcal{J}_2\mathcal{C}_0\right) +O\Big(X^{3/4+\varepsilon} Y\Big),
\end{multline*}
where for $j=0,1,2$, $\mathcal{C}_j$ and $\mathcal{J}_j$ are defined in equations \eqref{C2main}-\eqref{C0main} and \eqref{J}, respectively.
\end{lemma}
Consequently, equation \eqref{split in main + error} together with Lemma \ref{final SkE}, Lemma \ref{main term von} and Lemma \ref{main term id} proves Theorem \ref{thm1} and Theorem \ref{thm2}.

\medskip
\printbibliography
\end{document}